\documentclass[12pt]{amsart}
\usepackage{amssymb, verbatim}
\usepackage{enumerate}
\usepackage{color}

\date{\today}

\newtheorem{theorem}{Theorem}[section]

\newtheorem{lem}[theorem]{Lemma}

\theoremstyle{definition}
\newtheorem{definition}[theorem]{Definition}

\theoremstyle{remark}
\newtheorem{rem}{Remark}

\numberwithin{equation}{section}

\newcommand{\qf}[3]{\text{$(#1;\,#2)_{#3}$}}
\newcommand{\qnum}[1]{\text{$[#1]_{q}$}}
\newcommand{\qbinom}[2]{\text{$\genfrac[]{0pt}{}{#1}{#2}_q$}}

\newcommand{\qhyper}[6]{\,{}_{#1}\phi_{#2}\left(\!\!%
            \begin{array}{cc}{#3}\\[-0.1ex]{#4} \end{array}
            \Big|\,{#5};{#6}\right)}

\newcommand{\hyper}[5]{\,{}_{#1}F_{#2}\left(\!\!%
\begin{array}{cc}{\displaystyle{#3}}\\[-0.1ex]
{\displaystyle{#4}} \end{array}\Big|  \,{\displaystyle{#5}}
\right)}
\newcommand{\feriet}[8]{\,{}F^{#1}_{#2}\left(\!\!%
\begin{array}{c|c}
\begin{array}{cc}{\displaystyle{#3}:\displaystyle{#4}}\\[-0.1ex]
{\displaystyle{#5}:\displaystyle{#6}} \end{array} &
\,{\displaystyle{#7},\displaystyle{#8}}
\end{array} \right)}

\newcommand{\qferiet}[8]{\,{}\Phi^{#1}_{#2}\left[\!\!%
\begin{array}{c|c}
\begin{array}{cc}{\displaystyle{#3}:\displaystyle{#4}}\\[-0.1ex]
{\displaystyle{#5}:\displaystyle{#6}} \end{array} &\,
\begin{array}{cc}
\,{\displaystyle{#7}}\\[-0.1ex]
{\displaystyle{#8}}
\end{array}
\end{array} \right]}

\usepackage[T1]{fontenc}

\author[Area]{I. Area}
\author[Atakishiyev]{N. Atakishiyev}
\author[Godoy]{E. Godoy}
\author[Rodal]{J. Rodal}

\address[Area]{Departamento de Matem\'{a}tica Aplicada II, E.E. Telecomunicaci\'{o}n, 
Universidade de Vigo, 36310-Vigo, Spain.}\email[Area]{area@uvigo.es}

\address[Atakishiyev]{Instituto de Matem\'aticas, Unidad Cuernavaca, Universidad 
Nacional Aut\'o\-no\-ma de M\'exico, C.P. 62251 Cuernavaca, Morelos, M\'exico}
\email[Atakishiyev]{natig@matcuer.unam.mx}

\address[Godoy]{Departamento de Matem\'{a}tica Aplicada II, E.E. Industrial, 
Universidade de Vigo, 36310-Vigo, Spain.}\email[Godoy]{egodoy@dma.uvigo.es}

\address[Rodal]{Departamento de Matem\'atica Aplicada II, E.E. Telecomunicaci\'on, 
Universidade de Vigo, 36310-Vigo, Spain.}\email[Rodal]{jrodal@edu.xunta.es}

\begin{document}

\title[Bivariate $q$-orthogonal polynomials on $q$-linear lattices]
{Linear partial $q$-difference equations \\ on $q$-linear lattices and 
their bivariate \\ $q$-orthogonal polynomial solutions}

\begin{abstract}
Orthogonal polynomial solutions of an admissible potentially self-adjoint linear second-order 
partial $q$-difference equation of the hypergeometric type in two variables on $q$-linear 
lattices are analyzed. A  $q$-Pearson's system for the orthogonality weight function, as 
well as for the difference derivatives of the solutions are presented, giving rise to a solution 
of the $q$-difference equation under study in terms of a Rodrigues-type formula. The 
monic orthogonal polynomial solutions are treated in detail, giving explicit formulae
for the matrices in the corresponding recurrence relations they sa\-tisfy.
Lewanowicz and Wo\'zny [S. Lewanowicz, P. Wo\'zny, J. Comput. Appl. Math. 233 (2010) 1554--1561] 
have recently introduced  a  (non-monic) bivariate extension of big $q$-Jacobi polynomials together 
with a partial $q$-difference equation of the hypergeometric type that governs them. This equation 
is analyzed in the last section: we provide two more orthogonal polynomial solutions, namely, a second 
non-monic solution from the Rod\-rigues' representation, and the monic solution both from the recurrence 
relation that govern them and also explicitly given in terms of generalized bivariate basic hypergeometric 
series.  Limit relations as $q \uparrow 1$ for the partial $q$-difference equation and for the all three 
$q$-orthogonal polynomial solutions are also presented.
\end{abstract}

\maketitle

\section{Introduction}

As recalled in \cite{MR1100286}, Hahn \cite{MR0030647} considered the operator
\begin{equation}\label{eq:hahnoperator}
L_{q,\omega} f(x)=(D_{q,\omega}f)(x)= \frac{f(qx+\omega)-f(x)}{(q-1)x+\omega}\,, 
\quad x \in {\mathbf{R}} \setminus \left \{ \frac{\omega}{1-q}\right\},
\end{equation}
for all $q \in {\mathbf{R}} \setminus \{-1,0\}$, $\omega \in {\mathbf{R}}$ and $(q,\omega) \neq (1,0)$,
which for $q=1$ becomes the finite difference operator $\Delta_{\omega}$, and if $q \neq 1$ and $\omega=0$ 
then $L_{q,0}$ is the $q$-derivative operator \cite[Eq.~(2.3)]{MR0030647}
\[
L_{q,0}f(x)=(D_q \,f)(x)=\frac{f(qx)-f(x)}{(q-1)x}, \quad x\ne 0,  \quad q \neq 1,
\]
and $(D_q \,f)(0):=f'(0)$ by continuity, provided $f'(0)$ exists. Note that $\lim_{q \uparrow 1} (D_q \, f)(x)=f'(x)$ 
if $f$ is differentiable.

Hahn seems to have been the first to realize that the characterizations of classical orthogonal 
polynomial sequences (OPS) based on derivatives and differential equations are too restrictive 
\cite{MR1342384}. He posed and solved the following problems: find all OPS such that one of 
the following holds
\begin{enumerate}
\item $\{ L_{q,0}P_{n}(x) \}$ is also OPS;
\item $\{P_{n}(x)$\} satisfy the functional equation
\[
\sigma(x) L_{q,0}^{2} P_{n}(x) + \tau(x) L_{q,0}P_{n}(x) + \lambda_{n} P_{n}(x)=0;
\]
\item $P_{n}(x)$ has the representation
\[
P_{n}(x)=\frac{1}{\varrho(x)} L_{q,0}^{n} \left\{ f_{0}(x)f_{1}(x) \cdots f_{n-1}(x) \varrho(x) \right \},
\]
where $f_{k}(x)=f_{k+1}(qx)$ ;
\item If $P_{n}(x)=\sum a_{nk} x^{k}$ then $a_{nk}/a_{n,k-1}$ is a rational function of $q^{n}$ and $q^{k}$ ;
\item The moments associated with $\{P_{n}(x)\}$ satisfy
\begin{equation}\label{rrm}
M_{n}=\frac{a+bq^{n}}{c+d q^{n}} M_{n-1}, \quad a\,d-b\,c \neq 0\,,
\end{equation}
where $M_{n}$ are either the power moments (moments against $x^{k}$)
or the generalized moments (moments against the $q$-shifted factorial $\qf{x}{q}{k}
=\prod_{j=0}^{k-1}(1-xq^{j})$).
\end{enumerate}

Hahn's investigation led him to the most general set of polynomials belonging to
this class \cite{MR1342384},  now called $q$-Hahn polynomials:
\begin{equation}\label{qhahn}
Q_{n}(x;a,b,N;q)=\qhyper{3}{2}{q^{-n},abq^{n+1},x}{aq,q^{-N}}{q}{q}, \quad n=0,1,\dots,N.
\end{equation}

Following the works of Nikiforov and Uvarov \cite{MR1149380,0378.33001,MR922041}, a review 
of the hypergeometric-type difference equation for a function $y(x(s))$ on a non-uniform lattice 
$x(s)$ has been given in \cite{MR1342384}. Note that the difference-derivatives of $y(x(s))$ 
satisfy similar (to the initial ones for $y(x(s))$) equations if and only if the lattice $x(s)$ has 
the form
\begin{align}
x(s)&=c_{1} q^{s} + c_{2} q^{-s} + c_{3}, \qquad \text{or} \label{qquadratic} \\
 x(s)&=c_{4}s^{2}+c_{5}s+c_{6}, \label{quadratic}
\end{align}
where $q$, $c_{1}$, $c_{2}$ and $c_{3}$ are constants. Depending on the particular choice of constants, 
the lattices are commonly referred to as
\begin{enumerate}
\item Linear lattices if we choose in (\ref{quadratic}) $c_{4}=0$ and $c_{2} \neq 0$.
\item Quadratic lattices if we choose in (\ref{quadratic}) $c_{4} \neq 0$.
\item $q$-linear lattices (or $q$-exponential lattices) if we choose in (\ref{qquadratic}) $c_{2}=0$ 
and $c_{1} \neq 0$.
\item $q$-quadratic lattices if we choose in (\ref{qquadratic}) $c_{1}c_{2} \neq 0$.
\end{enumerate}

The Askey tableau of hypergeometric orthogonal polynomials contains the classical orthogonal polynomials, 
which can be written in terms of a hypergeometric function, starting at the top with Wilson and Racah 
polynomials on quadratic lattices and ending at the bottom with Hermite polynomials \cite{MR838967}. 
Hahn \cite{MR0030647}  actually studied the $q$-analogue of this scheme. So, there are $q$-analogues 
of all the families in the Askey tableau, often several $q$-analogues for one classical family. The master 
class of all these $q$-analogues is formed by the Askey-Wilson polynomials \cite{MR783216} on the 
$q$-quadratic lattice $x(s)=(q^{s}+q^{-s})/2$, which contain all other families as special or limit 
cases \cite{MR838970}. In \cite{MR1431306} Koornwinder gave a $q$-Hahn tableau: a $q$-analogue 
of the part of the Askey tableau dominated by $q$-Hahn polynomials,  in the $q$-linear lattice $x(s)=q^{s}$. 
Basic hypergeometric functions and $q$-orthogonal polynomials for arbitrary (including complex) values 
of $q$ are connected with quantum algebras and groups \cite{MR1371383}.

Recently, Koekoek, Lesky and Swarttouw  \cite{MR2656096} presented a classification 
of all families of classical orthogonal polynomials and their $q$-analogues, the classical 
$q$-orthogonal polynomials, as orthogonal polynomial solutions of the 
eigenvalue problem
\begin{equation}\label{eq:eqq1}
\phi(x) (D_{q,w})^{2}\,y_{n}(x) + \psi(x) (D_{q,w})y_{n}(x) = \lambda_{n} y_{n}(qx+w),
\end{equation}
where $D_{q,w}$ is the Hahn's operator (\ref{eq:hahnoperator}), $\phi(x)$ is a polynomial 
of at most degree $2$, $\psi(x)$ is a polynomial of exact  degree $1$, and $\lambda_{n}$ 
is the spectral parameter.

Besides well-known three-term recurrences, that $q$-orthogonal polynomial solutions of the 
latter equation do satisfy \cite{MR0481884,MR922041,MR0372517}, these solutions can be 
characterized in a number of ways, e.g. $k$-th $q$-derivatives of each family are again 
orthogonal and belong to the same family \cite{MR1100286,MR922041}. Moreover,  the 
orthogonality weight functions satisfy $q$-Pearson  equations \cite{MR1545034,MR1149380}, 
giving rise to Rodrigues' formulae \cite{MR1100286,MR922041} for the corresponding ortho\-gonal 
polynomials and their derivatives of any order. Also, the orthogonal polynomials posess a number 
of algebraic and $q$-difference properties, expressed  as $q$-derivative re\-presentations 
\cite{MR1100286,MR2241592,MR1545034}  and  structure relations \cite{MR1100286,MR0481884,
MR2345243}. The list of the above references is not exhaustive but only indicative for the kind 
of references that could be consulted on this topic.

It is quite remarkable that in these classical settings the coefficients, appearing in all of the 
aforementioned  algebraic and differential characterizations, can be explicitly computed in 
terms of the polynomial coefficients $\phi(x)$ and $\psi(x)$ of the hypergeometric-type 
$q$-difference equation (\ref{eq:eqq1}) \cite{MR2656096,MR922041}, which 
governs those $q$-classical families.

Orthogonal polynomials in several variables have been analyzed since a long time ago
\cite{hermite1864} and we refer to the books of Suetin \cite{MR1717891} and Dunkl 
and Xu \cite{MR1827871}, as basic references on this topic. Various multivariate 
extensions have been used in many applications such as image description and pattern 
recognition \cite{MR2929905}, or ternary drug mixtures \cite{Maher200895}, among 
others.

In 1991 Tratnik introduced some multivariable extensions of univariate orthogonal polynomials 
(see \cite{MR1123596,MR1122519} and references therein). Moreover, $q$-analogues of 
these systems have been constructed by Gasper and Rahman \cite{MR2132464,MR2132465,
MR2281173}, yielding systems of multivariable orthogonal Askey-Wilson polynomials and 
their special and limit cases. Bispectrality of multivariable Racah-Wilson and Askey-Wilson 
polynomials has been studied in \cite{MR2608419} and \cite{MR2784425}, respectively.

As indicated in \cite{MR2608419},  a beautiful extension of univariate orthogonal polynomials 
to the multivariate case is exemplified by symmetric Macdonald-Koornwinder polynomials, see, 
for instance, \cite{MR918416,MR1199128,MR1354144,MR1411136}.

In more recent papers the second-order linear partial differential equations of the hypergeometric type 
\cite{MR2853206} and their discretization on uniform lattices \cite{AGR2012,AG2013,MR2289245,MR2390890}, 
as well as a general way of introducing orthogonal polynomial families in two discrete variables on the simplex 
\cite{MR2117368}, have been analyzed. Therefore, it is possible to generalize the univariate classical 
orthogonal polynomials to the bivariate and multivariate versions by requiring that they obey a second-order 
partial differential equation of the hypergeometric type (continuous case) \cite{MR0228920,MR1717891}, or 
a second-order partial difference equation of the hypergeometric type (discrete case), as indicated before. 
Thus, the ``continuous'' polynomials can be analyzed as limits of the ``discrete'' ones \cite{AG2013}. 
Likewise, the corresponding differential operator will appear as a scaling limit of an appropriate difference 
operator, and the continuous distribution (the weight measure for the bivariate continuous polynomials) 
is obtained through a scaling limit from the discrete  distribution (the weight for the bivariate discrete 
polynomials).

The main goal of this paper is to extend the latter results on continuous and discrete bivariate cases 
to an admissible potentially self-adjoint linear second-order partial $q$-difference equation of the 
hypergeometric type on particular non-uniform lattices, and to study their orthogonal polynomial 
solutions.

The paper is organized as follows. In Section \ref{SEC:2} the linear second-order partial $q$-difference 
equations of the hypergeometric type are introduced, giving explicitly the coefficients of the equation for 
the partial $q$-derivatives  (of arbitrary order) of any solution  in terms of the coefficients of the initial 
equation. In Section \ref{S:3} we study the admissibility conditions for partial $q$-difference equations 
of the hypergeometric type. Next, in Section 4, the partial $q$-difference equation is written in the 
self-adjoint form,  which gives a number of useful identities for the orthogonality weight function for 
the polynomial solutions ($q$-Pearson's system), as well as of the $q$-difference derivatives of the polynomial 
solutions. The key point is to determine the orthogonality weight function from the polynomial coefficients 
of the initial equation, which is also explicitly worked out. In the remaining part of the paper we deal with 
admissible potentially self-adjoint linear second-order partial $q$-difference equations of the hypergeometric 
type. In Section \ref{S:5} an analogue of the well-known Rodrigues' formula for classical 
orthogonal polynomials is presented for orthogonal polynomial solutions of the partial $q$-difference equation. 
The monic orthogonal polynomial solutions of the partial $q$-difference equation are analyzed in detail in 
Section \ref{S:6}, where we give explicitly the matrices of the corresponding three-term recurrence relations 
for the most general equation, which  belongs to the class  under study.  Section \ref{section:example} is 
related with  bivariate big $q$-Jacobi polynomials and a partial $q$-difference equation, that govern them  
\cite{MR2559345}. Two novel bivariate $q$-orthogonal polynomial solutions of this equation are explicitly given. 
The first (non-monic) one is constructed from the Rodrigues' representation, derived in Section \ref{S:5}. 
The second novel (monic) solution of the equation is obtained from the general analysis, given in 
Section \ref{S:6}, i.e. by employing in this particular case the matrices of the three-term recurrence 
relations for the vector column of polynomials. Besides, this monic solution is also explicitly given in terms 
of generalized bivariate basic hypergeometric series.  Finally, limit  relations as $q \uparrow 1$  for the 
partial $q$-difference equation, as well as for the three above-mentioned solutions, are analyzed in detail.

\section{A linear second-order partial $q$-difference equation of the hypergeometric type}\label{SEC:2}

In what follows we shall assume that $0<q<1$ and we shall consider disconnected or independent 
non-uniform lattices \cite{MR1149380} as
\begin{equation}\label{independent}
x= x{(s)}=q^{s},\quad y= y{(t)}=q^{t}.
\end{equation}
Related with these $q$-linear lattices (\ref{independent}), let us introduce the following partial $q$-difference operators 
\begin{gather}
D_{q}^{1}f(x,y)=\frac{f(qx,y)-f(x,y)}{(q-1)x}, \quad D_{q}^{2}f(x,y)=\frac{f(x,qy)-f(x,y)}{(q-1)y}, 
\label{eq:defdq}\\
D_{q^{-1}}^{1}f(x,y)=\frac{q \left(f(x,y)-f\left({x}/{q},y\right)\right)}{(q-1) x}, \quad
D_{q^{-1}}^{2}f(x,y)=\frac{q \left(f(x,y)-f\left(x,{y}/{q}\right)\right)}{(q-1) y}. \label{eq:defdqm1}
\end{gather}
The rules for the partial $q$-derivatives of a product of two functions
$f(x,y)$ and $g(x,y)$ are given by
\begin{align}
D_{q}^{1}{(f g)}(x,y)&=f(x,y) D_{q}^{1}{g}(x,y) + g(qx,y) D_{q}^{1}{f}(x,y), \label{E:DQP1}\\
D_{q}^{2}{(f g)}(x,y)&=f(x,y) D_{q}^{2}{g}(x,y) + g(x,qy)
D_{q}^{2}{g}(x,y). \label{E:DQP2}
\end{align}

The following readily verified relations will also be used
\begin{equation}\label{EC:OPCON}
\begin{cases}
D_{q}^{1}D_{q^{-1}}^{2}{f}(x,y)=D_{q^{-1}}^{2}D_{q}^{1}{f}(x,y) ,
\qquad
\displaystyle{D_{q}^{2}D_{q^{-1}}^{1}{f}(x,y)=D_{q^{-1}}^{1}D_{q}^{2}{f}(x,y)} , \\
\displaystyle{D_{q}^{1}D_{q}^{2}{f}(x,y)=D_{q}^{2}D_{q}^{1}{f}(x,y) , \qquad
D_{q^{-1}}^{1}D_{q^{-1}}^{2}{f}(x,y)=D_{q^{-1}}^{2}D_{q^{-1}}^{1}{f}(x,y) }, \\
\displaystyle{D_{q^{-1}}^{1}{f}(x,y) = D_{q}^{1}{f}(x,y) + (1-q)x D_{q}^{1}D_{q^{-1}}^{1}{f}(x,y)} , \\
\displaystyle{D_{q^{-1}}^{2}{f}(x,y) = D_{q}^{2}{f}(x,y) + (1-q)y D_{q}^{2}D_{q^{-1}}^{2}{f}(x,y)} .
\end{cases}
\end{equation}

The following linear second-order partial differential equation has been considered in  \cite{MR2853206,0765.33009}
\begin{multline}\label{E:5J1}
\tilde{a}_{11}(x,y)\frac{\partial^2 u(x,y)}{\partial^2 x }
+\tilde{a}_{12}(x,y)\frac{\partial^2 u(x,y)}{\partial x \partial y}
+\tilde{a}_{22}(x,y)\frac{\partial^2 u(x,y)}{\partial^2 y } \\
+\tilde{b}_{1}(x,y)\frac{\partial u(x,y)}{\partial x}
+\tilde{b}_{2}(x,y)\frac{\partial u(x,y)}{\partial y}  + \lambda u(x,y) =0.
\end{multline}

Among many methods of approximating (\ref{E:5J1}), we shall discuss a linear partial $q$-difference
equation, obtained from (\ref{E:5J1}) via the simplest $q$-difference schemes of the second-order 
precision \cite{NIST,tesisjaime}:
\begin{multline}\label{EC:5J7}
a_{11}(x,y) \sqrt{q} D_{q}^{1} D_{q^{-1}}^{1} u(x,y)+ a_{22}(x,y) \sqrt{q} D_{q}^{2} D_{q^{-1}}^{2} u(x,y)\\
+{a}_{12a}(x,y) D_{q}^{1} D_{q}^{2} u(x,y)
+{a}_{12d}(x,y) D_{q^{-1}}^{1} D_{q^{-1}}^{2} u(x,y)\\
+ b_{1}(x,y) D_{q}^{1} u(x,y) + b_{2}(x,y) D_{q}^{2} u(x,y) + \lambda u(x,y)=0.
\end{multline}
It is important to note here that from the cross second partial derivative we have obtained two 
second-order $q$-difference operators. As is shown below through an example (see Section 
\ref{section:example}), the associated  polynomial coefficients $a_{12a}(x,y)$ and $a_{12d}(x,y)$ 
can be distinct.

\begin{definition}
We shall refer to
\[
u^{(k,\ell)}(x,y):=[D_{q}^{1}]^{(k)} [D_{q}^{2}]^{(\ell)} u(x,y)= D_{q}^{1} 
\stackrel{k)}{\cdots} D_{q}^{1} D_{q}^{2} \stackrel{\ell)}{\cdots} D_q^{2} u(x,y)
\]
as generalized difference of order $(k,\ell)$ for the function $u(x,y)$. 

\end{definition}

\begin{definition} We shall say that equation (\ref{EC:5J7}) is a partial $q$-difference equation of the 
hypergeometric type if all the generalized differences $u^{(k,\ell)}(x,y)$  for any solution $u=u(x,y)$ 
of (\ref{EC:5J7}) are also solutions of  equations  of the same type.
\end{definition}

In a similar way as Lyskova \cite{0765.33009} introduced the so-called basic class in the 
continuous case, we have:

\begin{lem}\label{L:5J100} Equation (\ref{EC:5J7}) is a partial $q$-difference equation of the 
hypergeometric type if and only if it has the form
\begin{multline}\label{eq:29n}
q  \left( a_{1}x^{2}+b_{1}x+c_{1} \right) D_{q}^{1} D_{q^{-1}}^{1} u(x,y)+ 
q \left( a_{2}y^{2}+b_{2}y+c_{2} \right) D_{q}^{2} D_{q^{-1}}^{2} u(x,y)\\
+\left(a_{3a} x y + b_{3a} x + c_{3a} y+ d_{3a}\right) D_{q}^{1} D_{q}^{2} u(x,y) \\
+\left( a_{3d} x y + b_{3d} x + c_{3d} y+ d_{3d} \right) D_{q^{-1}}^{1} D_{q^{-1}}^{2} u(x,y)\\
+ \left(f_{1}x+g_{1} \right) D_{q}^{1} u(x,y) + \left(f_{2}y+g_{2} \right) D_{q}^{2} u(x,y) 
+ \lambda u(x,y)=0,
\end{multline}
that is,
\begin{align*}
a_{11}(x,y)&=a_{11}(x)=\sqrt{q} \left( a_{1}x^{2}+b_{1}x+c_{1} \right), \quad a_{22}(x,y)=a_{2}(y)
=\sqrt{q} \left( a_{2}y^{2}+b_{2}y+c_{2} \right), \\
a_{12a}(x,y)&=a_{3a} x y + b_{3a} x + c_{3a} y+ d_{3a}, \quad
a_{12d}(x,y)=a_{3d} x y + b_{3d} x + c_{3d} y+ d_{3d}, \\
b_{1}(x,y)&=b_{1}(x)=f_{1}x+g_{1}, \quad b_{2}(x,y)=b_{2}(y)=f_{2}y+g_{2}.
\end{align*}
\end{lem}
\begin{proof}
Apply the operator $D_{q}^{1}$  to (\ref{EC:5J7}) in order to reveal  that the lemma is 
simply a consequence of the following seven partial results, based on the relations 
(\ref{E:DQP1})--(\ref{EC:OPCON}):
\begin{multline*}\label{eq:28}
(1) \quad D_{q}^{1} [a_{11}(x,y) q^{1/2} D_{q}^{1} D_{q^{-1}}^{1} u(x,y)]\\
=a_{11}(x,y) q^{-1/2} D_{q}^{1} D_{q^{-1}}^{1}  u^{(1,0)}(x,y)+D_{q}^{1} a_{11}(x,y) 
q^{-1/2}  D_{q}^{1} u^{(1,0)}(x,y);
\end{multline*}
\begin{equation*}
(2) \quad D_{q}^{1} [a_{22}(x,y) q^{1/2} D_{q}^{2}D_{q^{-1}}^{2} u(x,y)]=
a_{22}(x,y) q^{1/2} D_{q}^{2} D_{q^{-1}}^{2} u^{(1,0)}(x,y),  \hspace*{3cm}
\end{equation*}
provided that $a_{22}(x,y)$ does not depend on $x$  (in order to preserve the same 
structure of the equation);
\begin{multline*}
(3) \quad D_{q}^{1} [a_{12a}(x,y) D_{q}^{1} D_{q}^{2} u(x,y)]\\
=D_{q}^{1} a_{12a} (x,y) D_{q}^{2} u^{(1,0)} (x,y)+a_{12a}(qx,y) D_{q}^{1} D_{q}^{2} u^{(1,0)}(x,y);
\end{multline*}
\begin{multline*}
(4) \quad D_{q}^{1}[a_{12d}(x,y) D_{q^{-1}}^{1} D_{q^{-1}}^{2} u(x,y)]\\
=D_{q}^{1} a_{12d}(x,y) D_{q^{-1}}^{2} u^{(1,0)}(x,y)+a_{12d}(x,y)q^{-1/2} q^{-1/2}
D_{q^{-1}}^{1} D_{q^{-1}}^{2} u^{(1,0)}(x,y);
\end{multline*}
\begin{equation*}
(5) \quad D_{q}^{1} [b_{1}(x,y) D_{q}^{1}u(x,y)]
=b_{1}(qx,y)q^{1/2} q^{-1/2}D_{q}^{1} u^{(1,0)}(x,y)+ D_{q}^{1}b_{1}(x,y) u^{(1,0)}(x,y);
\end{equation*}
\[
(6) \quad D_{q}^{1}[b_{2}(x,y) D_{q}^{2}u(x,y)]=b_{2}(x,y) D_{q}^{2} u^{(1,0)}(x,y), \hspace*{6cm}
\]
provided that $b_{2}(x,y)$ does not depend on $x$; 

and finally,
\[
(7) \quad D_{q}^{1}[\lambda u(x,y)]=\lambda u^{(1,0)}(x,y). \hspace*{9cm}
\]

Repeating this process $k$ times in $x$ and $\ell$ times in $y$,  one obtains the following partial 
$q$-difference equation for the generalized difference of order $(k,\ell)$ of the function $u(x,y)$:
\begin{multline}\label{eq:newdeq}
a^{(k,\ell)}_{11}(x) \sqrt{q} D_{q}^{1} D_{q^{-1}}^{1} u^{(k,\ell)}(x,y)+a^{(k,\ell)}_{22}(y) \sqrt{q} D_{q}^{2} D_{q^{-1}}^{2} u^{(k,\ell)}(x,y)\\
+a^{(k,\ell)}_{12a}(x,y)  D_{q}^{1}D_{q}^{2}
u^{(k,\ell)}(x,y)
+a^{(k,\ell)}_{12d}(x,y)  D_{q^{-1}}^{1} D_{q^{-1}}^{2} u^{(k,\ell)}(x,y) \\
+b^{(k,\ell)}_{1}(x)  D_{q}^{1} u^{(k,\ell)}(x,y)
+b^{(k,\ell)}_{2}(y)  D_{q}^{2} u^{(k,\ell)}(x,y)+
\mu^{(k,\ell)} u^{(k,\ell)}(x,y)=0,
\end{multline}
where 
\begin{equation*}
\begin{cases}
a^{(k+1,\ell)}_{11}(x)=q^{-1} a^{(k,\ell)}_{11}(x)\,,\quad
a^{(k+1,\ell)}_{22}(y)=a^{(k,\ell)}_{22}{(y)}+q^{-1/2} (1-q) y D_{q}^{1} a^{(k,\ell)}_{12d}(x,y)\,,\\
a^{(k+1,\ell)}_{12a}(x,y)= a^{(k,\ell)}_{12a}(qx,y)\,,\quad
a^{(k+1,\ell)}_{12d}(x,y)=q^{-1} a^{(k,\ell)}_{12d}(x,y) \,,\\
b^{(k+1,\ell)}_{1}(x)=b^{(k,\ell)}_{1}{(qx)}+q^{-1/2} D_{q}^{1} a^{(k,\ell)}_{11}(x)\,,\\
b^{(k+1,\ell)}_{2}(y)=b^{(k,\ell)}_{2}(y)+ D_{q}^{1} a^{(k,\ell)}_{12a}(x,y) + D_{q}^{1} a^{(k,\ell)}_{12d}(x,y) \,,\\
\mu^{(k+1,\ell)}=\mu^{(k,\ell)}+ D_{q}^{1} b^{(k,\ell)}_{1}(x),
\end{cases}
\end{equation*}
and
\begin{equation*}
\begin{cases}
a^{(k,\ell+1)}_{11}(x)=a^{(k,\ell)}_{11}(x)+q^{-1/2} (1-q) x D_{q}^{2} a^{(k,\ell)}_{12d}(x,y)\,,\quad
a^{(k,\ell+1)}_{22}(y)=q^{-1} a^{(k,\ell)}_{22}(y)\,,\\
a^{(k,\ell+1)}_{12a}(x,y)=a^{(k,\ell)}_{12a}(x,qy)\,,\quad
a^{(k,\ell+1)}_{12d}(x,y)=q^{-1} a^{(k,\ell)}_{12d}(x,y)\,,\\
b^{(k,\ell+1)}_{1}(x)=b^{(k,\ell)}_{1}(x)+  D_{q}^{2} a^{(k,\ell)}_{12a}(x,y) + D_{q}^{2} a^{(k,\ell)}_{12d}(x,y)  \,,\\
b^{(k,\ell+1)}_{2}(y)=
b^{(k,\ell)}_{2}(qy) +q^{-1/2} D_{q}^{2} a^{(k,\ell)}_{22}(y)\,\\
\mu^{(k,\ell+1)}=\mu^{(k,\ell)}+ D_{q}^{2} b^{(k,\ell)}_{2}(y).
\end{cases}
\end{equation*}

If one computes the action of $D_{q}^{1}D_{q}^{2}=D_{q}^{2}D_{q}^{1}$ on the equation (\ref{EC:5J7}), 
then one obtains that
\[
D_{q}^{1}D_{q}^{2} a_{12i}(x,y)=0,
\]
or equivalently, the polynomials $a_{12a}(x,y)$ and $a_{12d}(x,y)$ should not contain the terms $x^{2}$ and $y^{2}$.

It is not hard to prove by induction that 
\begin{align} \label{a11rs}
a_{11}^{(k,\ell)}(x,y)&=\frac{ a_{11}(x,y)}{q^{k}} + \frac{(1-q^{\ell})x (c_{3d}+a_{3d}x)}
{q^{k+\ell-1/2}},\quad a_{12a}^{(k,\ell)}(x,y)=a_{12a}(q^{k}x,q^{\ell}y), \\
a_{22}^{(k,\ell)}(x,y)&=\frac{ a_{22}(x,y)}{q^{\ell}} + \frac{(1-q^{k})y(b_{3d}+a_{3d}y)}{q^{k+\ell-1/2}},\quad
a_{12d}^{(k,\ell)}(x,y)=q^{-(k+\ell)} a_{12d}(x,y),
\end{align} 
\begin{align}
b_{1}^{(k,\ell)}(x,y)&=  b_{1}(q^{k}x,q^{\ell}y) +
\frac{(1-q)\qnum{k}\qnum{\ell} (c_{3d}+(a_{3d}-a_{3a}q^{k+\ell-1})x)}{q^{k+\ell-1}}, \\ & +\frac{\qnum{k} (b_{1}+a_{1}(q^{k}+1)x)}{q^{k-1}}  + \frac{\qnum{\ell}(c_{3d}+a_{3d}x+q^{\ell-1}(c_{3a}+a_{3a}x))}{q^{\ell-1}} \nonumber \\
b_{2}^{(k,\ell)}(x,y)&=b_{2}(q^{k} x,q^{\ell} y) + \frac{(1-q)\qnum{k}\qnum{\ell} (b_{3d}+(a_{3d}-a_{3a}q^{k+\ell-1})y)}{q^{k+\ell-1}} \label{b2rs} \\
& + \frac{\qnum{\ell}(b_{2}+a_{2}(q^{\ell}+1)y)}{q^{\ell-1}} +
\frac{\qnum{k}(b_{3d}+a_{3d}y+q^{k-1}(b_{3a}+a_{3a}y))}{q^{k-1}}, \nonumber \\
\mu^{(k,\ell)}&=\lambda+\frac{\qnum{k}(f_{1}q^{k-2}+a_{1}\qnum{k-1})}{q^{k-2}} +
\frac{\qnum{\ell}(f_{2}q^{\ell-2}+a_{2} \qnum{\ell-1})}{q^{\ell-2}}\\&+
\frac{\qnum{k}\qnum{\ell} (a_{3d}+a_{31}q^{k+\ell-2})}{q^{k+\ell-2}}, \nonumber
\end{align}
where the $q$-number is
\begin{equation}\label{eq:qnumber}
\qnum{z}=\frac{q^{z}-1}{q-1}, \qquad z \in {\mathbb{C}}.
\end{equation}

\end{proof}

\section{Admissible equations}\label{S:3}
\begin{definition}
The partial $q$-difference equation of the hypergeometric type (\ref{EC:5J7}) will be called {admissible} 
if there exists an infinite sequence $\{\lambda_n\}$ ($n=0,1, \dots$) such that for each $\lambda=\lambda_n$, 
there are precisely $n+1$ linearly independent  polynomial solutions of total degree $n$ and no non-trivial 
solutions in the form of polynomials, whose total degree is less than $n$.
\end{definition}
This concept was introduced by Krall and Sheffer \cite{MR0228920} in the case of second-order partial 
differential equations and also by Y. Xu in \cite[Section 2]{MR2132628} for the case of second-order partial 
difference equations (without assuming that equations are of the hypergeometric type),  and analyzed 
later on in \cite{MR2853206} and \cite{AGR2012,MR2289245,tesisjaime}, for the continuous and discrete 
cases, respectively.

In the case $n=0$, the equation (\ref{EC:5J7}) also implies that a non-trivial solution can only 
exist when $\lambda_0=0$.

Observe that this definition of admissibility of equation (\ref{EC:5J7}) implies that all numbers
\[
\lambda_0=0,\lambda_1, \lambda_2, \dots, \lambda_n, \dots,
\]
are distinct,  $\lambda_m \neq \lambda_n,\,\, m \neq n$.

From Lemma \ref{L:5J100} one can deduce

\begin{theorem}
The partial $q$-difference equation of the hypergeometric type (\ref{eq:29n}) is admissible if and only if
\begin{equation}\label{EC:AAAAAA}
f_{2}=f_{1}, \quad a_{2}=a_{1}, \quad a_{3a}=a_{1}q+f_{1}(q-1), \quad a_{3d}=a_{1},
\end{equation}
and
\begin{equation}\label{EC:AUTOVALOR}
\lambda_{n}=-\qnum{n}\,\left(f_{1}-a_{1} q \qnum{1-n} \right),
\end{equation}
and the numbers $a_{1}$ and $f_{1}$ are such that for any non-negative integer $m$ the 
following condition holds
\[
f_{1}-a_{1} q \qnum{1-m} \neq 0.
\]
\end{theorem}
\begin{proof}
A proof can be given in a similar way as in \cite[pp. 93--97]{MR1717891} for the multivariate 
continuous situation.
\end{proof}

It is therefore plain that with the notations of Lemma \ref{L:5J100}  the equation  (\ref{eq:29n}) 
can be written as
\begin{multline}\label{eq:29nn}
q  \left( a_{1}x^{2}+b_{1}x+c_{1} \right) D_{q}^{1} D_{q^{-1}}^{1} u(x,y)+ 
q \left( a_{1}y^{2}+b_{2}y+c_{2} \right) D_{q}^{2} D_{q^{-1}}^{2} u(x,y)\\
+\left((a_{1}q+f_{1}(q-1)) x y + b_{3a} x + c_{3a} y+ d_{3a}\right) D_{q}^{1} D_{q}^{2} u(x,y) \\
+\left( a_{1} x y + b_{3d} x + c_{3d} y+ d_{3d} \right) D_{q^{-1}}^{1} D_{q^{-1}}^{2} u(x,y)\\
+ \left(f_{1}x+g_{1} \right) D_{q}^{1} u(x,y) + \left(f_{1}y+g_{2} \right) D_{q}^{2} u(x,y) 
+ \lambda_{n} u(x,y)=0,
\end{multline}
i.e. 
\begin{equation}\label{sigmas}
\begin{cases}
a_{11}(x)=\sqrt{q} \left( a_{1}x^{2}+b_{1}x+c_{1} \right),  \quad b_{1}(x)=f_{1}x+g_{1}, \quad
b_{2}(y)=f_{1}y+g_{2}, \\
a_{12a}(x,y)=(a_{1}q+f_{1}(q-1)) x y + b_{3a} x + c_{3a} y+ d_{3a}, \\
a_{22}(y)=\sqrt{q} \left( a_{1}y^{2}+b_{2}y+c_{2} \right), \quad
a_{12d}(x,y)=a_{1} x y + b_{3d} x + c_{3d} y+ d_{3d}.
\end{cases}
\end{equation}

\section{Potentially self-adjoint operator}\label{S:4}

From the admissible linear second-order partial $q$-difference equation of the hypergeo\-metric 
type  (\ref{eq:29nn})  we introduce the following second-order partial $q$-difference operator:
\begin{multline}\label{EC:h1}
\mathcal{D}_{q}[f(x,y)]=
a_{11}(x)\sqrt{q} D_{q}^{1}D_{q^{-1}}^{1}{f}(x,y)+a_{22}(y)\sqrt{q} D_{q}^{2}D_{q^{-1}}^{2}{f}(x,y)\\
+a_{12a}(x,y)D_{q}^{1}D_{q}^{2}{f}(x,y)+a_{12d}(x,y)D_{q^{-1}}^{1}D_{q^{-1}}^{2}{f}(x,y) \\
+ b_{1}(x)D_{q}^{1}{f}(x,y) + b_{2}(y) D_{q}^{2}{f}(x,y).
\end{multline}
This enables us to write  (\ref{eq:29nn}) as
\begin{equation}\label{EC:h2}
\mathcal{D}_{q}f(x,y)+ \lambda_{n} f(x,y) =0.
\end{equation}

\begin{lem}
The adjoint operator $\mathcal{D}_{q}^{\dagger}$ of $\mathcal{D}_{q}$, defined by (\ref{EC:h1}), is given by
\begin{multline}\label{adj222}
\mathcal{D}_{q}^{\dagger}[f(x,y)]=
\sqrt{q} D_{q}^{1}D_{q^{-1}}^{1}(a_{11}(x){f}(x,y))+\sqrt{q} D_{q}^{2}D_{q^{-1}}^{2}(a_{22}(y){f}(x,y))\\
q^{2}D_{q}^{1}D_{q}^{2}(a_{12d}(x,y){f}(x,y))+\frac{1}{q^2}D_{q^{-1}}^{1}D_{q^{-1}}^{2}(a_{12a}(x,y){f}(x,y)) \\
- \frac{1}{q} D_{q^{-1}}^{1}(b_{1}(x){f}(x,y)) -\frac{1}{q}D_{q^{-1}}^{2}(b_{2}(y){f}(x,y)).
\end{multline}
\end{lem}
\begin{proof}
The result is a direct consequence of
\[
\left[D^{i}_{q} \right]^{\dagger}=-\frac{1}{q}D^{i}_{q^{-1}}, \quad 
\left[D^{i}_{q^{-1}} \right]^{\dagger}=-q D^{i}_{q}, \quad i=1,2.
\]
\end{proof}

\begin{definition}\label{def:psa}
The operator ${\mathcal{D}}_{q}$ is potentially self-adjoint in a domain $G$, if there exists a 
positive real function $\varrho(\textbf{x}) = \varrho(x,y)$  in this domain, such that the operator 
$\varrho(\textbf{x}) {\mathcal{D}}_{q}$ is self-adjoint in the domain $G$, i.e., $(\varrho(\textbf{x}) {\mathcal{D}}_{q})^{\dag}=\varrho(\textbf{x}) {\mathcal{D}}_{q}$ (see \cite[Chapter V]{MR1717891}).
\end{definition}
In order that ${\mathcal{D}}_{q}$ be potentially self-adjoint, we multiply (\ref{eq:29nn}) through 
by a positive function $\varrho(\textbf{x})=\varrho(x,y)$ in some domain $G$, to be chosen later, 
to arrive at 
\begin{multline}\label{SA:11}
a_{11}(x)\varrho(x,y) \sqrt{q} D_{q}^{1}D_{q^{-1}}^{1}{f}(x,y)+a_{22}(y)\varrho(x,y) 
\sqrt{q} D_{q}^{2}D_{q^{-1}}^{2}{f}(x,y)\\
+a_{12a}(x,y)\varrho(x,y) D_{q}^{1}D_{q}^{2}{f}(x,y)+a_{12d}(x,y)\varrho(x,y) 
D_{q^{-1}}^{1}D_{q^{-1}}^{2}{f}(x,y) \\
+ b_{1}(x)\varrho(x,y) D_{q}^{1}{f}(x,y) +b_{2}(y)\varrho(x,y)D_{q}^{2}{f}(x,y)
+\lambda\varrho(x,y)f(x,y)=0,
\end{multline}
which can be written in the self-adjoint form if the following $q$-Pearson's system of equations 
is satisfied:
\begin{equation}\label{eq:pearson}
\begin{cases}
\varrho(x,y)a_{12a}(x,y)=q^{2}\varrho(qx,qy)a_{12d}(qx,qy), \\
\varrho(x,y)\phi_{1}(x,y)=\varrho(qx,y)\omega_{1}(qx,y),\\
\varrho(x,y)\phi_{2}(x,y)=\varrho(x,qy)\omega_{2}(x,qy),
\end{cases}
\end{equation}
where
\begin{equation}\label{eq:46}
\begin{cases}
\omega_{1}(qx,y)=\sqrt{q}ya_{11}(qx)-xq^{2}a_{12d}(qx,y),\\
\omega_{2}(x,qy)=\sqrt{q}xa_{22}(qy)-yq^{2}a_{12d}(x,qy),\\
\phi_{1}(x,y)=\sqrt{q}y a_{11}(x)-xa_{12a}(x,y)+(q-1)xyb_1(x),\\
\phi_{2}(x,y)=\sqrt{q}xa_{22}(y)-ya_{12a}(x,y)+(q-1)xyb_2(y).
\end{cases}
\end{equation}

The $q$-Pearson's system (\ref{eq:pearson}) can be also written as
\begin{equation}\label{EC:5J780}
\begin{cases}
\sqrt{q}D_{q}^{1}(\varrho(x,y) a_{11}(x))+q^{-1}D_{q^{-1}}^{2}(\varrho(x,y) a_{12a}(x,y))
=\varrho(x,y) b_1(x,y),\\
\sqrt{q}D_{q}^{2}(\varrho(x,y) a_{22}(y))+q^{-1}D_{q^{-1}}^{1}(\varrho(x,y) a_{12a}(x,y))
=\varrho(x,y) b_2(x,y),\\
q^{4}D_{q}^{1}D_{q}^{2}(\varrho(x,y) a_{12d}(x,y))
=D_{q^{-1}}^{1}D_{q^{-1}}^{2}(\varrho(x,y) a_{12a}(x,y)),
\end{cases}
\end{equation}
or equivalently,
\begin{equation}\label{EC:5J790}
\begin{cases}
D_{q}^{1}(\omega_{1}(x,y)\varrho(x,y))=\frac{1}{q}D_{q^{-1}}^{1}(\phi_{1}(x,y)\varrho(x,y)),\\
D_{q}^{2}(\omega_{2}(x,y) \varrho(x,y))=\frac{1}{q}D_{q^{-1}}^{2}(\phi_{2}(x,y)\varrho(x,y)),\\
q^{4}D_{q}^{1}D_{q}^{2}(\varrho(x,y)a_{12d}(x,y))=D_{q^{-1}}^{1}D_{q^{-1}}^{2}(\varrho(x,y)a_{12a}(x,y)).
\end{cases}
\end{equation}

\subsection{Computation of the weight function}\label{calculopeso}
By introducing the functions
\begin{equation}\label{g1g2b}
\mathcal{G}_{1}(x,y)=\frac{\phi_{1}(x,y)}{\omega_{1}(qx,y)}, \qquad \mathcal{G}_{2}(x,y)
=\frac{\phi_{2}(x,y)}{\omega_{2}(x,qy)},
\end{equation}
where $\phi_{j}(x,y)$ and $\omega_{j}(x,y)$ are defined in (\ref{eq:46}), $j=1,2$, and using 
the $q$-Pearson's system (\ref{eq:pearson}),  one obtains that
\begin{align}
&\varrho(qx,y)=\mathcal{G}_{1}(x,y)\varrho(x,y),\quad \varrho(x,qy)=\mathcal{G}_{2}(x,y)\varrho(x,y) \label{g1g2}, \\
& q^{2}\mathcal{G}_{1}(x,y) \mathcal{G}_{2}(qx,y) a_{12d}(qx,qy)= a_{12a}(x,y) = q^{2} \mathcal{G}_{1}(x,qy)\mathcal{G}_{2}(x,y)a_{12d}(qx,qy),  \label{eq:nova3} \\
&y \mathcal{G}_{2}(x,y)D^{2}_{q}(\mathcal{G}_{1}(x,y))=x\mathcal{G}_{1}(x,y)D^{1}_{q}(\mathcal{G}_{2}(x,y)). \label{eq:key}
\end{align}
From (\ref{g1g2}) it follows then that 
\[
\frac{\ln \left[\varrho(qx,y) \right]-\ln \left[\varrho(x,y) \right]}{(q-1)x}
=\frac{\ln \left[\mathcal{G}_{1}(x,y) \right]}{(q-1)x}\,,
\]
or, equivalently,
\[
D^{1}_{q}\left[\ln(\varrho(x,y)) \right]=\ln \left[(\mathcal{G}_{1}(x,y))^{{1}/{((q-1)x)}} \right],
\]
and therefore
\begin{equation}\label{faltaba1}
D^{1}_{q}\left[\ln(\varrho(x,y)) \right] - D_{q}^{1}  \left[\ln\varrho(x,y_{0}) \right]=\frac{1}{(q-1)x}
\ln \left[\frac{\mathcal{G}_{1}(x,y) }{\mathcal{G}_{1}(x,y_{0}) }\right].
\end{equation}

Upon using the $q$-integral due to J. Thomae \cite{Thomae1969} and F.H. Jackson \cite{JACKSON1910}  (see also \cite{MR2128719,MR2191786,MR2656096}), this yields
\[
\int^{x}_{x_{0}}D^{1}_{q}\left[ \ln(\varrho(s,y)) \right]d_{q}s=\ln \left[\varrho(x,y) 
\right]-\ln\left[\varrho(x_{0},y) \right],
\]
and
\begin{multline}\label{faltaba3}
\ln \left[\varrho(x,y) \right]-\ln\left[\varrho(x_{0},y) \right]=\int^{x}_{x_{0}} \ln \left[(\mathcal{G}_{1}(s,y))^{{1}/{((q-1)s)}}\right] d_{q}s\\
=(1-q)x\sum^{\infty}_{j=0}q^{j}\ln \left[(\mathcal{G}_{1}(q^{j}x,y))^{{1}/{((q-1)q^{j}x)}} \right]
-(1-q)x_{0}\sum^{\infty}_{j=0}q^{j}\ln \left[(\mathcal{G}_{1}(q^{j}x_{0},y))^{{1}/{((q-1)q^{j}x_{0})}} \right]\\
=\sum^{\infty}_{j=0}\ln \left[\frac{\mathcal{G}_{1}(q^{j}x_{0},y)}{\mathcal{G}_{1}(q^{j}x,y)}\right]+c_{1}(y).
\end{multline}
In a similar way, one has
\begin{equation}\label{faltaba2}
\ln \left[\varrho(x,y) \right]-\ln \left[\varrho(x,y_{0}) \right]=\sum^{\infty}_{j=0}\ln\left[\frac{\mathcal{G}_{2}(x,q^{j}y_{0})}{\mathcal{G}_{2}(x,q^{j}y)} \right]+c_{2}(x).
\end{equation}

From (\ref{eq:key}) we deduce  that
\[
\frac{(q-1)xD^{1}_{q}(\mathcal{G}_{2}(x,yq^{j}))}{\mathcal{G}_{2}(x,yq^{j})}
=\frac{(q-1)y q^{j}D^{2}_{q}(\mathcal{G}_{1}(x,yq^{j}))}{\mathcal{G}_{1}(x,yq^{j})}.
\]
Upon using
\[
(q-1)xD^{1}_{q} \left(\ln \vert f \vert \right)=\ln \left[\frac{(q-1)xD^{1}_{q}f(x,y)}{f(x,y)}+1 \right],
\]
and then applying the  operator $D_{q}^{1}$ to (\ref{faltaba2}), from (\ref{faltaba1}) one obtains that 
\begin{multline*}
\ln \left[\frac{\mathcal{G}_{1}(x,y) }{\mathcal{G}_{1}(x,y_{0}) }\right]
\\=\sum^{\infty}_{j=0}(q-1)xD^{1}_{q}[\ln(\mathcal{G}_{2}(x,q^{j}y_{0}))]-(q-1)xD^{1}_{q}[\ln(\mathcal{G}_{2}(x,q^{j}y))]+(q-1)xD^{1}_{q}(c_{2}(x))\\
=\sum^{\infty}_{j=0}\ln[\frac{(q-1)xD^{1}_{q}\mathcal{G}_{2}(x,q^{j}y_{0})}{\mathcal{G}_{2}(x,q^{j}y_{0})}+1]-\ln[\frac{(q-1)xD^{1}_{q}\mathcal{G}_{2}(x,q^{j}y)}{\mathcal{G}_{2}(x,q^{j}y)}+1]+(q-1)xD^{1}_{q}(c_{2}(x))\\
=\sum^{\infty}_{j=0}\ln[\frac{(q-1)y_{0}q^{j}D^{2}_{q}\mathcal{G}_{1}(x,q^{j}y_{0})}{\mathcal{G}_{1}(x,q^{j}y_{0})}+1]-\ln[\frac{(q-1)yq^{j}D^{2}_{q}\mathcal{G}_{1}(x,q^{j}y)}{\mathcal{G}_{1}(x,q^{j}y)}+1]
+(q-1)xD^{1}_{q}(c_{2}(x))\\
=\sum^{\infty}_{j=0}(q-1)y_{0}q^{j}D^{2}_{q}(\ln(\mathcal{G}_{1}(x,q^{j}y_{0})))-(q-1)yq^{j}D^{2}_{q}
(\ln(\mathcal{G}_{1}(x,q^{j}y)))+(q-1)xD^{1}_{q}(c_{2}(x))\\
=\int^{y}_{y_{0}}D^{2}_{q}(\ln(\mathcal{G}_{1}(x,t)))d_{q}t+(q-1)xD^{1}_{q}(c_{2}(x))\\
=\ln(\mathcal{G}_{1}(x,y))-\ln(\mathcal{G}_{1}(x,y_{0}))+(q-1)xD^{1}_{q}(c_{2}(x)),
\end{multline*}
and therefore
\[
(q-1)xD^{1}_{q}(c_{2}(x))=0,
\]
which implies that $c_{2}(x)=c_{3}$ is a constant. In a similar way one verifies that $c_{1}(y)=c_{4}$ 
is also constant.

Substituting in (\ref{faltaba3}) and (\ref{faltaba2}),
\[
\ln \left[\varrho(x,y) \right]-\ln\left[\varrho(x_{0},y) \right]
=\sum^{\infty}_{j=0}\ln \left[\frac{\mathcal{G}_{1}(q^{j}x_{0},y)}{\mathcal{G}_{1}(q^{j}x,y)}\right]+c_{4},
\]
\[
\ln \left[\varrho(x,y) \right]-\ln \left[\varrho(x,y_{0}) \right]=\sum^{\infty}_{j=0}
\ln\left[\frac{\mathcal{G}_{2}(x,q^{j}y_{0})}{\mathcal{G}_{2}(x,q^{j}y)} \right]+c_{3},
\]
we finally obtain (up to a multiplicative constant) the explicit expression for the weight function 
solution of the $q$-Pearson's  system of equations (\ref{eq:pearson}),
\begin{equation}\label{eq:weight}
\varrho(x,y)=\prod^{\infty}_{j=0} 
\frac{\mathcal{G}_{1}(q^{j}x_{0},y)\,\mathcal{G}_{2}(x_{0},q^{j}y_{0})}
{\mathcal{G}_{1}(q^{j}x,y)\,\mathcal{G}_{2}(x_{0},q^{j}y)},
\end{equation}
where $\mathcal{G}_{1}(x,y)$ and $\mathcal{G}_{2}(x,y)$ are defined in (\ref{g1g2b}).

In a similar way  one can obtain the following representation for the orthogonality weight
function, associated with the $q$-derivatives of any order: 
\begin{equation}\label{eq:rhokl}
\varrho^{(k,\ell)}(x,y)=
\prod^{\infty}_{j=0} 
\frac{\mathcal{G}_{1}^{(k,\ell)}(q^{j}x_{0},y)\,\mathcal{G}_{2}^{(k,\ell)}(x_{0},q^{j}y_{0})}
{\mathcal{G}_{1}^{(k,\ell)}(q^{j}x,y)\,\mathcal{G}_{2}^{(k,\ell)}(x_{0},q^{j}y)}.
\end{equation}
Here ${\mathcal{G}}^{(k,\ell)}_{1}(x,y)$ and ${\mathcal{G}}^{(k,\ell)}_{2}(x,y)$ are defined by 
inserting into (\ref{g1g2b}) the polynomial coefficients $a_{11}^{(k,\ell)}(x,y)$, $a_{22}^{(k,\ell)}(x,y)$, 
$a_{12a}^{(k,\ell)}(x,y)$, $a_{12d}^{(k,\ell)}(x,y)$, $b_{1}^{(k,\ell)}(x,y)$, and $b_{2}^{(k,\ell)}(x,y)$,  
introduced in Section \ref{SEC:2} and given explicitly in terms of the coefficients of the initial equation 
(\ref{eq:29nn}) in (\ref{a11rs})--(\ref{b2rs}). It is important to note here that, for example, $\varrho^{(1,1)}(x,y)$ 
can be computed in two ways: as the $D_{q}^{1}$ derivative of the $D_{q}^{2}$ derivative or vice versa. 
The following relation ensures that one arrives at the same result:
\begin{equation}\label{eq:coupling}
\omega^{(k,\ell+1)}_{1}(qx,y)\omega^{(k,\ell)}_{2}(qx,qy)=\omega^{(k+1,\ell)}_{2}(x,qy)
\omega^{(k,\ell)}_{1}(qx,qy), \qquad k,\ell \geq 0.
\end{equation}
We shall refer to the latter equation as the coupling hypergeometric condition, analogous 
to \cite[Eq. (50)]{MR2289245}.

\section{Rodrigues' formula}\label{S:5}

Rodrigues' formula for classical orthogonal polynomials in one variable is an important tool for 
analyzing the fundamental properties of these polynomials \cite{MR1561893,MR0259197,MR1149380}. 
The great advantage of the Rodrigues' formula is its form as $n$th derivative of the ortho\-gonality
weight function. In \cite{MR1717891}, an analogue of the Rodrigues' formula for orthogonal polynomials 
over a domain in two variables, which are solutions of admissible and potentially self-adjoint equations, is 
presented. Kwon et al. \cite{MR1851313} succeeded in deriving a (functional) Rodrigues-type formula 
for multivariable orthogonal polynomial solutions of a second-order partial differential equation. 
In recent papers appropriate Rodrigues' formulae for polynomials solutions of second-order admissible, 
hypergeometric and potentially self-adjoint partial differential and difference equations have been 
presented \cite{MR2853206,MR2289245}.

By using the results of the previous sections in a similar vein as was elaborated by Suetin 
\cite[Theorem 3, p. 151]{MR1717891} for the continuous case, it is not hard to arrive at 
an explicit expression for a polynomial solution of an admissible potentially self-adjoint 
second-order partial $q$-difference equation of the hypergeometric type (\ref{eq:29nn}). 
The expression
\begin{multline}\label{eq:rodrigues}
\tilde{P}_{n,m}(x,y)=\frac{\Lambda_{n,m}}{\varrho(x,y)} [D_{q^{-1}}^{1}]^{(n)} [D_{q^{-1}}^{2}]^{(m)} 
\left[ \varrho^{(n,m)}(x,y) \right]
\\ =\frac{q^{n(1-n)/2+m(1-m)/2}\Lambda_{n,m}}{\varrho(x,y)} [D_{q}^{1}]^{(n)} [D_{q}^{2}]^{(m)} 
\left[ \varrho^{(n,m)}(q^{-n}x,q^{-m}y) \right] \\
=\frac{q^{n(1-n)/2+m(1-m)/2}\Lambda_{n,m}}{\varrho(x,y)} [D_{q}^{1}]^{(n)} [D_{q}^{2}]^{(m)} 
\left[ \varrho(x,y) \prod_{k=0}^{n-1} \omega_{1}(q^{-k} x,y) \prod_{s=0}^{m-1} \omega_{2}(x,q^{-s}y) \right],
\end{multline}
defines an algebraic polynomial of total degree $n + m$ in the variables $x$ and $y$, called 
Rod\-rigues' formula for the bivariate $q$-orthogonal polynomials $\tilde{P}_{n,m}(x, y)$, 
that are solutions of (\ref{eq:29nn}). In (\ref{eq:rodrigues}) the $\Lambda_{n,m}$ are 
normalizing constants, $\varrho(x,y)$ and $\varrho^{(n,m)}(x,y)$ are defined by 
(\ref{eq:weight}) and (\ref{eq:rhokl}), respectively, and $\omega_{1}(x,y)$ and 
$\omega_{2}(x,y)$ are defined in (\ref{eq:46}). In the limit as $q$ tends to 1, the Rodrigues' 
formula (\ref{eq:rodrigues}) reduces to the one derived in \cite{MR1717891} for the continuous 
case. Moreover, in the bivariate discrete case a Rodrigues' formula  has been also given in 
\cite{MR2289245} upon employing the same approach as in \cite{MR1717891}.

\section{Monic orthogonal polynomial solutions}\label{S:6}

One essential difference between polynomials in one variable and in several variables is the lack 
of an obvious basis in the latter \cite{MR1827871}. One possibility to avoid this problem is to consider 
graded lexicographical order and use the matrix vector representation, first introduced by Kowalski 
\cite{MR647129,MR647128} and later on studied by Xu \cite{MR1215438,MR1169912}. 

Let $\textbf{x}=(x,y)\in\mathbb{R}^2$, and let $\textbf{x}^n$ ($n\in \mathbb{N}_0$) denote 
the column vector of the monomials $x^{n-k} y^{k}$, whose elements are arranged in graded 
lexicographical order (see \cite[p. 32]{MR1827871}):
\begin{equation}\label{MONO}
\textbf{x}^n= (x^{n-k}y^{k})\,,\quad 0 \leq k \leq n, \quad  n\in \mathbb{N}_0\,.
\end{equation}
Let $\{P_{n-k,k}^n(x,y)\}$ be a sequence of polynomials in the space $\Pi_n^2$ of all polynomials 
of total degree at most $n$ in two variables, $\textbf{x}=(x,y)$, with real coefficients. Such 
polynomials are finite sums of terms of the form $ax^{n-k}y^{k}$, where $a \in \mathbb{R}$.

From now on ${\mathbb{P}}_n$ will denote the (column) polynomial vector
\begin{equation}\label{defPn}
{\mathbb{P}}_n= (P_{n,0}^n(x,y), P_{n-1,1}^n(x,y), \dots,P_{1,n-1}^n(x,y), P_{0,n}^n(x,y))^\text{T}.
\end{equation}
Then, each polynomial vector ${\mathbb{P}}_n$ can be written in terms of the basis (\ref{MONO}) as:
\begin{equation}\label{EXPP}
{\mathbb{P}}_n= G_{n,n}\textbf{x}^n+ G_{n,n-1}\textbf{x}^{n-1}+ \dots + G_{n,0}\,\textbf{x}^0,
\end{equation}
where $G_{n,j}$ are  $(n+1)\times(j+1)$-matrices  and $G_{n,n}$ is a nonsingular square matrix 
of the size $(n+1)\times(n+1)$.

A polynomial vector $\widehat{\mathbb{P}}_n$  is said to be monic if its leading matrix coefficient 
$\widehat{G}_{n,n}$ is the identity matrix (of the size $(n+1)\times(n+1)$), that is,
\begin{equation}\label{EXPPMonico}
\widehat{\mathbb{P}}_n= \textbf{x}^n+ \widehat{G}_{n,n-1}\textbf{x}^{n-1}+ \dots 
+ \widehat{G}_{n,0}\,\textbf{x}^0\,.
\end{equation}
Then each of its polynomial entries $\widehat{P}_{n-k,k}^n(x,y)$ are of the form:
\begin{equation}\label{CompoMonico}
\widehat{P}_{n-k,k}^n(x,y) = x^{n-k} y^{k} + \text{terms of lower total degree}\,.
\end{equation}
In what follows the ``hat'' notation $\widehat{\mathbb{P}}_n$ will represent monic polynomials.

The existence of a recurrence relation for any vector of bivariate discrete orthogonal polynomial 
family can be established in more general settings than those considered here \cite{MR2081045}. 
The following existence theorem, proved in \cite{MR1827871}, can be applied for infinite or finite 
($n=0,1,\dots,N$) sequences of polynomials (see examples 4.1 and 4.2 in \cite{MR2081045}), 
because we are using graded lexicographical order (\ref{MONO}).

\begin{theorem}
Let $\mathcal{L}$ be a positive definite moment linear functional acting on the space $\Pi_n^2$ 
of all polynomials of total degree at most $n$ in two variables, and $\{\mathbb{P}_{n}\}_{n \geq 0}$ 
be an orthogonal family with respect to $\mathcal{L}$. Then, for $n \geq 0$, there exist unique matrices 
$A_{n,j}$ of the size $(n+1) \times (n+2)$, $B_{n,j}$  of the size $(n+1) \times (n+1)$, and $C_{n,j}$ 
of the size $(n+1) \times n$, such that
\begin{equation}\label{RRTT}
x_j\mathbb{P}_n=A_{n,j}\mathbb{P}_{n+1} + B_{n,j}\mathbb{P}_{n} + C_{n,j}\mathbb{P}_{n-1}, 
\quad j =1, 2,
\end{equation}
with the initial conditions $\mathbb{P}_{-1}=0$ and $\mathbb{P}_{0}=1$. Here the notation 
$x_{1}=x$ and $x_{2}=y$ is used.
\end{theorem}
In this section we give explicit expressions for the matrices $A_{n,j}$, $B_{n,j}$ and $C_{n,j}$,
which appear in the three-term recurrence relations (\ref{RRTT}), in terms of the coefficients of 
$a_{ii}$, $a_{12j}$ and $b_{i}$  in (\ref{eq:29nn}). These  matrices enable one to compute the 
monic orthogonal polynomial solutions of an admissible potentially self-adjoint second-order partial 
$q$-difference equation of the hypergeometric type.

The weight function (\ref{eq:weight}) determines the moment linear functional $\mathcal{L}$, defined in 
the space $\Pi_{n}^{2}$ of all polynomials of total degree at most $n$ in two variables, in terms 
of a double $q$-integral
\[
\mathcal{L}(P)=\iint_{G} P(x,y) \varrho(x,y) d_{q}x d_{q}y,
\]
 in an appropriate domain $G \subset \mathbb{R}^2$, which can be applied to polynomial vectors. 
Thus, in what follows $\{\widehat{\mathbb{P}}_n\}_{n\in\mathbb{N}_0}$ denotes a monic vector 
polynomial family  solution of ~(\ref{eq:29nn}), that is  orthogonal with respect to $\varrho(x,y)$, 
\begin{equation}\label{ortorho}
\mathcal{L}(\mathbf{x}^m\,\, \widehat{\mathbb{P}}_n^{\,\,T})=\iint_{G} \mathbf{x}^m
\,\, \widehat{\mathbb{P}}_n^{\,\,T} \varrho(x,y) d_{q}x d_{q}y\,=\,
\left\{\begin{array}{l}
0 \in \mathcal{M}^{(m+1,n+1)} \,,\,\, \mbox{if $n > m$,}\\[4mm]
H_n \in \mathcal{M}^{(n+1,n+1)} \,,\,\, \mbox{if $m=n$,}
\end{array}\right.
\end{equation}
where $H_n$ $($of the size $(n+1)\times (n+1)$$)$ is nonsingular.

Let us first introduce the matrices $L_{n,j}$ of the size $(n+1) \times (n+2)$
\begin{equation}\label{defLL}
\begin{array}{ll} L_{n,1}=\begin{pmatrix}
1 & & \text{\circle{10}}&0 \\&\ddots &&\vdots  \\&\text{\circle{10}}&1&0 \end{pmatrix}  &
\text{and} \quad
L_{n,2}=\begin{pmatrix} 0 & 1& &\text{\circle{10}}
\\\vdots &  &\ddots &  \\0&\text{\circle{10}}&&1 \end{pmatrix}
\end{array},
\end{equation}
so that
\begin{equation}\label{DEFLM}
x\, \textbf{x}^n=L_{n,1}\textbf{x}^{n+1}, \quad y\, \textbf{x}^n=L_{n,2}\textbf{x}^{n+1}.
\end{equation}
Observe that
\begin{gather}
x^2\, \textbf{x}^n =L_{n,1}L_{n+1,1}\textbf{x}^{n+2}, \quad y^2\, \textbf{x}^n =L_{n,2}L_{n+1,2}
\textbf{x}^{n+2}\,, \label{LLLL} \\
L_{n,2} L_{n+1,1} = L_{n,1}L_{n+1,2}, \nonumber
\end{gather}
and for $j=1,2$,
\begin{equation}\label{eq:L}
L_{n,j}\,L_{n,j}^{\text{T}} = I_{n+1},
\end{equation}
where $I_{n+1}$ denotes the identity matrix of the size $n+1$.

From the definition of the partial $q$-difference operators in (\ref{eq:defdq}) and (\ref{eq:defdqm1}), 
one obtains that
\begin{equation*}\label{DERMON}
\displaystyle{D_{q}^{j}\textbf{x}^n=\mathbb{E}_{n,j}\,\textbf{x}^{n-1} }  ,\quad
\displaystyle{D_{q^{-1}}^{j}\textbf{x}^n=\mathbb{K}_{n,j}\,\textbf{x}^{n-1} }  , \qquad j=1,2,
\end{equation*}
where the matrices $\mathbb{E}_{n,j}$ of the size $(n+1) \times n$ are given by
\begin{equation}\label{eq:mate}
\begin{array}{rr} \mathbb{E}_{n,1}=\begin{pmatrix}
 \qnum{n} & & &\text{\circle{10}} \\
  & \qnum{n-1}& &  \\
 &  &\ddots &  \\
 &\text{\circle{10}}& & 1 \\
 0 &\dots&0&0\end{pmatrix},  \qquad
\mathbb{E}_{n,2}=\begin{pmatrix} 0&\dots& &0 \\1 & &
&\text{\circle{10}}
\\ & \qnum{2} & &  \\& &\ddots &  \\&\text{\circle{10}}& & \qnum{n}
\end{pmatrix}\,,
\end{array}
\end{equation}
the matrices $\mathbb{K}_{n,j}$ of the size $(n+1) \times n$ are given by
\begin{equation}\label{eq:matk}
\mathbb{K}_{n,1}=\begin{pmatrix}
q^{1-n} \qnum{n} & & &\text{\circle{10}} \\
  & q^{2-n} \qnum{n-1}& &  \\
 &  &\ddots &  \\
 &\text{\circle{10}}& & 1 \\
 0 &\dots&0&0\end{pmatrix},  \qquad
\mathbb{K}_{n,2}=\begin{pmatrix} 0&\dots& &0 \\1 & &
&\text{\circle{10}}
\\ & q^{-1} \qnum{2} & &  \\& &\ddots &  \\&\text{\circle{10}}& & q^{1-n} \qnum{n}
\end{pmatrix}\,,
\end{equation}
and  the $q$-number is defined in (\ref{eq:qnumber}).

Substitute the expansion (\ref{EXPPMonico}) into (\ref{eq:29nn}) and then equate the coefficients 
of $\textbf{x}^{n-1}$ and $\textbf{x}^{n-2}$  to arrive at the following explicit expressions for 
the matrices $\widehat{G}_{n,n-1}$ and $\widehat{G}_{n,n-2}$:
\begin{align}
\widehat{G}_{n,n-1}&={\mathbb{S}}_{n} {\mathbb{F}}_{n-1}^{-1}(\lambda_{n}) , \label{GN1} \\
\widehat{G}_{n,n-2}&= \left( {\mathbb{T}}_{n} + \widehat{G}_{n,n-1} {\mathbb{S}}_{n-1} \right) 
{\mathbb{F}}_{n-2}^{-1}(\lambda_{n}), \label{GN2}
\end{align}
where the nonsingular matrix ${\mathbb{F}}_{n}(\lambda_{\ell})$ is given by
\begin{equation}\label{MATRIX:FF}
{\mathbb{F}}_{n}(\lambda_{\ell})=(\lambda_{n} - \lambda_{\ell}) {\mathbb{I}}_{n+1},
\end{equation}
$\lambda_{n}$ is given in (\ref{EC:AUTOVALOR}), ${\mathbb{I}}_{n+1}$ denotes the identity matrix 
of the size $(n+1) \times (n+1)$, and the matrix ${\mathbb{S}}_{n}$ of the size $(n+1)\times n$ is 
given in terms of the coefficients of the polynomials $a_{ii}$, $a_{12i}$ and $b_{i}$  from the partial 
$q$-difference equation (\ref{eq:29nn}), explicitly written down in (\ref{sigmas}), as
\begin{align}\label{COEFPDE}
{\mathbb{S}}_{n} &= \begin{pmatrix}
s_{1,1} & &  &  & \text{\circle{15}} \\
s_{2,1} & s_{2,2} & & &  \\
& \ddots & \ddots & &  \\
& & s_{n-1,n-2} & s_{n-1,n-1} & \\
& & & s_{n,n-1} & s_{n,n} \\
& \text{\circle{15}} & & 0 & s_{n+1,n}
\end{pmatrix}\quad (n\geq 1)\,.
\end{align}
Here, for $1 \leq i \leq n$,
\begin{align*}
s_{i,i}&=\qnum{n-i+1} \left( g_{1}+q^{1+i-n} b_{1} \qnum{n-i} + c_{3a} \qnum{i-1} 
+ q^{2-n} c_{3d} \qnum{i-1} \right) \,,\\
s_{i+1,i}&=\qnum{i-1} \left(b_2 q^{3-i} \qnum{i-2}+\qnum{n+1-i} \left(b_{3a}+b_{3d}   
q^{2-n}\right)+g_{2}\right)\,.
\end{align*}

Besides, the matrix ${\mathbb{T}}_{n}$ of the size $(n+1) \times (n-1)$, which appears in (\ref{GN2}), 
is given in terms of the coefficients of the partial $q$-difference equation (\ref{eq:29nn}) as
\begin{align}\label{COEFPDE2}
{\mathbb{T}}_{n} =d_{3a} {\mathbb{E}}_{n,2} {\mathbb{E}}_{n-1,1} + c_{1} q 
{\mathbb{K}}_{n,1} {\mathbb{E}}_{n-1,1} + c_{2} q {\mathbb{K}}_{n,2} 
{\mathbb{E}}_{n-1,2} + d_{3d} {\mathbb{K}}_{n,2} {\mathbb{K}}_{n-1,1} \quad (n\geq 2)\,,
\end{align}
where the matrices ${\mathbb{E}}_{n,i}$ and ${\mathbb{K}}_{n,i}$, $i=1,2$, are defined 
by  (\ref{eq:mate}) and (\ref{eq:matk}), respectively.

Now, in this monic situation, it is possible to generalize the well-known explicit expressions for the 
coefficients in the three-term recurrence relation for the one variable case \cite[p. 14]{MR1149380} 
to the $q$-bivariate case. This is achieved with the aid of the auxiliary matrices $L_{n,j}$, defined 
in (\ref{defLL}) and (\ref{DEFLM}), and the following result, proved in \cite{MR2853206} in the 
continuous bivariate situation and therefore valid also in this $q$-bivariate situation, because it 
is a consequence of the three-term recurrence relations (\ref{RRTT}).
\begin{theorem}\label{TTRRenG}
In the monic case, the explicit expressions of the matrices $A_{n,j}$, $B_{n,j}$ and $C_{n,j}$ ($j =1, 2$), 
that appear in (\ref{RRTT}) in terms of the values of the leading coefficients $\widehat{G}_{n,n-1}$ and 
$\widehat{G}_{n,n-2}$ (see (\ref{GN1}) and (\ref{GN2}), respectively), are given by
\begin{equation}\label{COEFTTR}
\begin{cases}
A_{n,j}=L_{n,j}, \quad n \geq 0,\\
B_{0,j}=-L_{0,j}\widehat{G}_{1,0}, \quad B_{n,j}=\widehat{G}_{n,n-1}L_{n-1,j}-L_{n,j} \widehat{G}_{n+1,n}, 
\quad n \geq 1,  \\
C_{1,j}=-(L_{1,j}\widehat{G}_{2,0}+B_{1,j}\widehat{G}_{1,0}), \\
C_{n,j}= \widehat{G}_{n,n-2}L_{n-2,j}-L_{n,j} \widehat{G}_{n+1,n-1}-B_{n,j} \widehat{G}_{n,n-1}, \quad n \geq 2\,,
\end{cases}
\end{equation}
where the matrices $L_{n,j}$ have been introduced in (\ref{defLL}).
\end{theorem}

It is of interest to remark here that, as it is described in \cite{MR1827871}, since
\begin{equation}
{\rm{rank}} (L_{n,j})=n+1= {\rm{rank}} (C_{n+1,j}), \quad j=1,2, \quad n \geq 0,
\end{equation}
the columns of the joint matrices
\[
L_n=\left(L_{n,1}^T \,,L_{n,2}^T \right)^T \quad \text{and}\quad C_n
=\left(C_{n,1}^T \,,C_{n,2}^T \right)^T,
\]
of the size $(2n+2)\times(n+2)$ and $(2n+2)\times n$, respectively, are linearly independent, 
that is,
\begin{equation}
{\rm{rank}} (L_{n})= n+2, \quad {\rm{rank}} (C_{n})= n.
\end{equation}
Therefore the matrix $L_{n}$ has full rank, so that there exists a unique matrix $D_n^{\dag}$ 
of the size $(n+2)\times(2n+2)$, called the generalized inverse of $L_{n}$,
\begin{equation}
D_n^{\dag}= \left(D_{n,1} |D_{n,2} \right)=\left(L_n^TL_n\right)^{-1}L_n^T,
\end{equation}
such that
\[
D_n^{\dag}L_{n}=I_{n+2}.
\]

Moreover, using the left inverse $D_n^{\dag}$ of the joint matrix $L_n$,
\[ D_n^{\dag}=\begin{pmatrix}
1 & & && 0& & &&
\\&1/2&&\bigcirc &1/2&& \bigcirc &&\\&&\ddots&
&&\ddots&&\\& \bigcirc && 1/2& \bigcirc  &&1/2&&\\&&&&0&&&1
\end{pmatrix},
\]
one can write a recursive formula for the monic orthogonal polynomials
\begin{equation}\label{RF}
\widehat{\mathbb{P}}_{n+1}=D_n^{\dag}\left[\begin{pmatrix} x
\\y \end{pmatrix}\otimes I_{n+1}-B_{n}\right]
{\widehat{\mathbb{P}}}_{n}-D_n^{\dag} C_{n}
\widehat{\mathbb{P}}_{n-1}, \quad n \geq 0,
\end{equation}
with the initial conditions $\widehat{\mathbb{P}}_{-1}=0$, $\widehat{\mathbb{P}}_{0}=1$. 
In (\ref{RF}) the symbol $\otimes$ denotes the Kronecker product and
\begin{equation}\label{JOMA}
B_n=\left(B_{n,1}^T \,,B_{n,2}^T \right)^T\,, \quad
C_n=\left(C_{n,1}^T \,,C_{n,2}^T \right)^T\,,
\end{equation}
are matrices of the size $(2n+2)\times(n+1)$ and $(2n+2)\times n$, respectively, which can be 
obtained by using  (\ref{COEFTTR}) in terms of the coefficients of the partial $q$-difference 
equation (\ref{eq:29nn}), explicitly given in (\ref{sigmas}). This means that the recurrence 
(\ref{RF}) gives another realisation of \cite[(3.2.10)]{MR1827871}, already appeared in the 
bivariate discrete case in \cite{MR2390890}.

Therefore, from (\ref{RF}) it is possible to compute a monic orthogonal polynomial solution 
of an admissible potentially self-adjoint linear second-order partial $q$-difference equation of the 
hypergeometric type (\ref{eq:29nn}).

\section{Illustrative example}\label{section:example}

In this section we discuss in detail an example, which is related to the admissible potentially 
self-adjoint linear second-order partial $q$-difference equation of the hypergeometric type, 
satisfied by a (non-monic) bivariate extension of big $q$-Jacobi polynomials, recently introduced 
in \cite{MR2559345,Lewanowicz20138790}. The monic orthogonal polynomial solutions are expressed 
by means of the three-term recurrence relations, that govern them, and also explicitly in terms 
of generalized bivariate basic hypergeometric series. Moreover, a third (non-monic) solution of 
the same partial $q$-difference equation  is provided by using the Rodrigues' re\-presentation 
(\ref{eq:rodrigues}). In the limit when $q \uparrow 1$  the partial $q$-difference equation reduces 
to a second-order partial differential equation of the hypergeometric type, with monic Appell polynomials 
as solutions. Limit relations between the three orthogonal polynomial solutions of the partial $q$-difference 
equation and corresponding orthogonal polynomial solutions of the partial differential equation are explicitly 
given. Besides, the corres\-ponding orthogonality weight functions are also linked by appropriate limit relations 
as $q \uparrow 1$.

{ \bf Example.} Lewanowicz and Wo\'zny  \cite{MR2559345} have recently introduced 
the following bivariate extension of the big $q$-Jacobi polynomials
\begin{multline}\label{eq:bbqj}
P_{n,k}(x,y;a,b,c,d;q):=P_{n-k}(y;a,bcq^{2k+1},dq^{k};q)
y^{k}(dq/y;q)_{k}P_{k}(x/y;c,b,d/y;q), \\ n\in
\mathbb{N}\,,\quad \,k=0,1,\dots,n, \quad q \in (0,1),\quad 0< aq,bq,cq <1,\quad d< 0,
\end{multline}
where the univariate big $q$-Jacobi polynomials \cite{MR838970}  (see also \cite[Eq. (19.5.1)]{MR2656096}) 
are defined by means of the $_3 \phi_{2}$ basic hypergeometric series as
\[
P_{m}(t;A,B,C;q):=\qhyper{3}{2}{q^{-m},\,ABq^{m+1},\,t}{Aq,\,Cq}{q}{q}, \,\, 0<A<1/q, \,\,\, 0<B<1/q, \,\,\, C<0.
\]

They have demonstrated that the polynomials (\ref{eq:bbqj}) satisfy a linear second-order 
partial $q$-difference equation  of the form  (\ref{eq:29nn}),  with coefficients given by
\begin{equation}\label{ste1}
\begin{cases}
a_{11}(x)= {\sqrt{q}}\,\left( d\,q - x \right) \,\left( a\,c\,q^2 - x \right) ,\quad a_{22}(y)
={\sqrt{q}}\,\left( a\,q - y \right) \,\left( d\,q - y \right), \\
a_{12a}(x,y)= a\,c\,q^4\,\left( d - b\,x \right) \,\left( 1 - y \right)  ,  \quad a_{12d}(x,y)
=\left( d\,q - x \right) \,\left( a\,q - y \right), \\[4mm]
\displaystyle{b_1(x)=\frac{q\,\left( q\,\left( d - a\,c\,d\,q^2 +
a\,c\,q\,\left( 1 + b\,q\,\left( -1 + x \right)  \right)  \right) - x \right) }{-1 + q}}  ,\\[4mm]
\displaystyle{b_2(y)=\frac{q\,\left( d\,q + a\,q\,\left( 1 - d\,q + b\,c\,q^2\,
\left( -1 + y \right)  \right)  - y \right) }{-1 + q}},\\
\displaystyle{\lambda_{n}=q^{2-n} \qnum{n} \frac{ \left(1-a b c q^{n+2}\right)}{q-1}}.
\end{cases}
\end{equation}
Thus this $q$-difference equation is admissible, potentially self-adjoint and of the hypergeometric type. 

In \cite{MR2559345} it has been proved that the polynomials (\ref{eq:bbqj}) satisfy the following 
orthogonality relation
\begin{multline}\label{ortstan}
\int_{dq}^{aq}\int_{dq}^{cqy}W(x,y;a,b,c,d;q)P_{n,k}(x,y;q)P_{m,l}(x,y;q)d_{q}xd_{q}y\\=H_{n,k}(a,b,c,d;q)\delta_{n,m}\delta_{k,l},
\end{multline}
where $0<aq,bq,cq<1$, $d<0$, and the weight function is defined by
\begin{equation}\label{pesostan}
W(x,y;a,b,c,d;q):=\frac{(dq/y,c^{-1}x/y,x/d,y/a,y/d;q)_{\infty}}{y(c^{-1}d/y,cqy/d,x/y,bx/d,y;q)_{\infty}}.
\end{equation}
Here the $q$-shifted factorial $\qf{a}{q}{k}$ is equal to
\[
\qf{a}{q}{0}=1, \qquad \qf{a}{q}{k}=\prod_{j=0}^{k-1}(1-a q^{j}), 
\qquad (k=1,2,\dots, \,\text{or}\,\, \infty).
\]
and we have employed the conventional notation 
\[
\qf{a_{1},\dots,a_{r}}{q}{k}=\qf{a_{1}}{q}{k} \cdots \qf{a_{r}}{q}{k}
\]
for products of $q$-shifted factorials.

It is easy to check that the $W(x,y;a,b,c,d;q)$ is a solution of the $q$-Pearson's system (\ref{eq:pearson}), which has been presented in a different form in \cite[Lemma 3.1]{Lewanowicz20138790} for this particular example. 
From the explicit expression for the weight function (\ref{eq:weight}), we have (up to a normalization 
constant), upon taking into account that  $x_{0}=dq$,
\begin{equation}\label{eq:610n}
\varrho(x,y)=\frac{
\qf{y/a,x/d,d q/y,x/(cy),y/d}{q}{\infty}
}
{
\qf{y,x/y,bx/d,dq/(cy),cy/d}{q}{\infty}
},
\end{equation}
which coincides with the weight function, given in (\ref{pesostan}), up to the positive multiplicative 
constant $-d/c$. Observe that equations (\ref{eq:nova3}), (\ref{eq:key}), and (\ref{eq:coupling}) 
also hold.

The partial $q$-difference (\ref{eq:29nn}) with polynomial coefficients (\ref{ste1}) has another 
(non-monic) orthogonal polynomial solution which can be computed from the Rodrigues' formula 
(\ref{eq:rodrigues}) as
\begin{multline}
\tilde{P}_{n,m}(x,y;a,b,c,d;q)
=\frac{\Lambda_{n,m}}{\varrho(x,y)} \\ \times[D_{q}^{1}]^{(n)} [D_{q}^{2}]^{(m)} 
\left[ \varrho(x,y) x^{2n}y^{2m} \qf{dq/x}{q}{n} \qf{aq/y}{q}{m} \qf{x/y}{q}{m} 
\qf{cqy/x}{q}{n} \right],
\end{multline}
where $\varrho(x,y)$ is given in (\ref{eq:610n}) and $\Lambda_{n,m}$ is a normalizing constant.

The partial $q$-difference equation (\ref{eq:29nn}), with coefficients given in (\ref{ste1}), has a third 
(monic) orthogonal polynomial solution, which can be computed recursively from Theorem \ref{TTRRenG}. 
From (\ref{COEFTTR}) it follows that the matrix $B_{n,1}$ has entries
\begin{equation}\label{eq:bn1}
\begin{cases}
\displaystyle{\frac{d q^{-i+n+2} \left(a c q^{2 i-1} \left(q^{- i+n+1} \left(b \left(a c q^{i+n+1}
+q^{-i+n+2}-q-1\right)-q-1\right)+1\right)+1\right)}
 {\left(a b c q^{2 n+1}-1\right) \left(a b c q^{2 n+3}-1\right)}} & \\[3mm]
\displaystyle{+\frac{a c q^{n+2} \left(-b (q+1) q^{n-i} \left(a c q^{2 i}+q\right)+a
   b (b+1) c q^{2 n+2}+b+1\right)}{\left(a b c q^{2 n+1}-1\right)
   \left(a b c q^{2 n+3}-1\right)}}, 
  \quad (i=j), \\[3mm]
\displaystyle{-\frac{\left(q^{i-1}-1\right) \left(a q^{i-1}-1\right) q^{-i+n+2}
   \left(a b c d q^{2 n+2}-a b c (q+1) q^{n+1}+d\right)}{\left(a b c
   q^{2 n+1}-1\right) \left(a b c q^{2 n+3}-1\right)}}, \quad (i=j+1), \\
   0, \quad \text{otherwise},
   \end{cases}
\end{equation}
the matrix $B_{n,2}$ has entries
\begin{equation}\label{eq:bn2}
\begin{cases}
\displaystyle{q^{-i} \left(\frac{q \qnum{i} \left(a q^i-1\right) \left(d q^{n+1}-1\right)}{a b c q^{2 n+3}-1}
+\frac{q \qnum{i-1}   \left(q-a q^i\right) \left(d q^n-1\right)}{a b c q^{2   n+1}-1}+q\right)}, \quad (i=j), \\[3mm]
\displaystyle{-\frac{a c q^{i+1} \left(q^{-i+n+1}-1\right) \left(b
   q^{-i+n+1}-1\right) \left(a b c q^{2 (n+1)}-d (q+1)
   q^n+1\right)}{\left(a b c q^{2 n+1}-1\right) \left(a b c q^{2
   n+3}-1\right)}}, \quad (i=j-1), \\[3mm]
   0, \quad \text{otherwise,}
\end{cases}
\end{equation}
the matrix $C_{n,1}$ has entries
\begin{equation}\label{eq:cn1}
\begin{cases}
\displaystyle{-\frac{a c q^{n+2} \left(d q^n-1\right) \left(q^{-i+n+1}-1\right)
   \left(b q^{-i+n+1}-1\right) }{\left(a b c q^{2
   n}-1\right) \left(a b c q^{2 n+1}-1\right)^2 \left(a b c q^{2
   n+2}-1\right)}} \\[3mm]
\times \displaystyle{\left(a c q^{i+n}-1\right) \left(a b c
   q^{n+1}-d\right) \left(a b c q^{i+n}-1\right)}, \quad (i=j), \\[4mm]
\displaystyle{-\frac{a c \left(q^{i-1}-1\right) \left(a q^{i-1}-1\right) \left(d
   q^n-1\right) q^{-i+2 n+3} \left(a b c q^{n+1}-d\right)}{\left(a b
   c q^{2 n}-1\right) \left(a b c q^{2 n+1}-1\right)^2 \left(a b c
   q^{2 n+2}-1\right)}}  \\[3mm]
\times \displaystyle{\left( -b (q+1) q^{-i+n-1} \left(a c q^{2 i}+q^2\right)+a b (b+1) c q^{2
   n+1}+b+1 \right) }, \quad (i=j+1), \\[3mm]
\displaystyle{-\frac{a b c \left(a q^i-q\right) \left(a q^i-q^2\right) \left(d
   q^n-1\right) q^{-2 i+3 n+2} \qf{q}{q}{i-1} \left(a b c
   q^{n+1}-d\right)}{\qf{q}{q}{i-3} \left(a b c q^{2 n}-1\right) \left(a
   b c q^{2 n+1}-1\right)^2 \left(a b c q^{2 n+2}-1\right)}}, \quad (i=j+2), \\[3mm]
   0, \quad \text{otherwise,}
\end{cases}
\end{equation}
and the matrix $C_{n,2}$ has entries
\begin{equation}\label{eq:cn2}
\begin{cases}
\displaystyle{-\frac{a c (q-1) \left(d q^n-1\right) q^{n-i} \qnum{-i+n+1}
   \left(b q^{n+1}-q^i\right) \left(a b c q^{n+1}-d\right)
  }{\left(a b c q^{2 n}-1\right)
   \left(a b c q^{2 n+1}-1\right)^2 \left(a b c q^{2 n+2}-1\right)}}, & \\[3mm]
   \displaystyle{\times \left((a+1) q^{i+1} \left(a b c q^{2 n+1}+1\right)-a b c (q+1)
   q^{2 n+2}-a (q+1) q^{2 i}\right)}, \quad (i=j), \\[3mm]
\displaystyle{-\frac{a \left(q^{i-1}-1\right) q^{n+1} \left(a q^{i-1}-1\right)
   \left(d q^n-1\right) \left(b c q^{-i+2 n+2}-1\right) }{\left(a b c
   q^{2 n}-1\right) \left(a b c q^{2 n+1}-1\right)^2 \left(a b c q^{2
   n+2}-1\right)}} \\[3mm]
   \times \displaystyle{\left(a b c
   q^{n+1}-d\right) \left(a b c q^{-i+2 n+2}-1\right)}, \quad (i=j+1), \\[3mm]
\displaystyle{-\frac{a^2 c^2 q^{n+2} \left(d q^n-1\right) \left(q^i-b q^n\right)
   \left(q^i-b q^{n+1}\right) \qf{q}{q}{n-i+1} \left(a b c
   q^{n+1}-d\right)}{\qf{q}{q}{n-i-1} \left(a b c q^{2 n}-1\right)
   \left(a b c q^{2 n+1}-1\right)^2 \left(a b c q^{2 n+2}-1\right)}}, \quad (i=j-1), \\[3mm]
   0, \quad \text{otherwise.}
\end{cases}
\end{equation}

From (\ref{RF}) each column polynomial vector (\ref{defPn}) can be obtained by using the 
above matrices $B_{n,i}$ and $C_{n,i}$, $i=1,2$.

\begin{rem}
In \cite{MR2559345} the authors obtained the matrices of the three-term recurrence relations (\ref{RRTT}),
satisfied by the non-monic solution (\ref{eq:bbqj}), by using the recurrence relation that governs  the 
univariate $q$-Jacobi polynomials. Notice that we have derived the matrices in the monic case from our 
approach given in Section \ref{S:6}, with different shapes as in the non-monic case.
\end{rem}

\begin{rem}
Observe that the monic polynomial solution $ \widehat{P}_{n,m}(x,y;a,b,c,d;q)$  of the partial 
$q$-difference equation (\ref{eq:29nn}), with coefficients, given in (\ref{ste1}) and obtained 
from Theorem \ref{TTRRenG} (with the matrix coefficients that are given above),  can be also 
written in terms of generalized bivariate basic hypergeometric series as
\begin{multline}\label{eq:solhyp}
 \widehat{P}_{n,m}(x,y;a,b,c,d;q)= \frac{
\left(\frac{d}{b}\right)^{n} \qf{aq}{q}{m} \qf{bq}{q}{n} \qf{d q^{n+1}}{q}{m} \qf{a b c q^{m+2}/d}{q}{n}
}
{
\qf{a b c q^{m+n+2}}{q}{n+m}
} \\
\times \sum_{i=0}^{n} \sum_{j=0}^{m} \frac{(-1)^{-i-j}  \qbinom{n}{i}
   \qbinom{m}{j} q^{\frac{1}{2} (i (i-2 n+1)+j (j-2 m+1))}
\qf{a b c q^{m+n+2}}{q}{i+j} 
}
{
\qf{aq}{q}{j} \qf{bq}{q}{i}
\qf{d q^{n+1}}{q}{j}   \qf{a b c q^{m+2}/d}{q}{i}
 }  \qf{bx/d}{q}{i} \qf{y}{q}{j} \\
= \frac{
\left(\frac{d}{b}\right)^{n} \qf{aq}{q}{m} \qf{bq}{q}{n} \qf{d q^{n+1}}{q}{m} \qf{a b c q^{m+2}/d}{q}{n}
}
{
\qf{a b c q^{m+n+2}}{q}{n+m}
} \\
\times
\qferiet{1:2;2}{0:2;2}{abcq^{n+m+2}}{q^{-n},bx/d;q^{-m},y}{-}{bq,abcq^{m+2}/d;aq,dq^{n+1}}{q:q,q}{0,0,0},
\end{multline}
where the generalized bivariate basic hypergeometric series is defined by \cite{MR834385}
\begin{multline*}
\qferiet{\lambda:r;s}{\mu:u;v}{\alpha_{1},\dots,\alpha_{\lambda}}{a_{1},\dots,a_{r};c_{1},\dots,c_{s}}
{\beta_{1},\dots,\beta_{\mu}}{b_{1},\dots,b_{u};d_{1},\dots,d_{v}}{q:x,y}{i,j,k}\\
 = \sum_{m,n=0}^{\infty} 
\frac{
\qf{\alpha_{1},\dots,\alpha_{\lambda}}{q}{m+n}
\qf{a_{1},\dots,a_{r}}{q}{m}
\qf{c_{1},\dots,c_{s}}{q}{n}
}
{
\qf{\beta_{1},\dots,\beta_{\mu}}{q}{m+n}
\qf{b_{1},\dots,b_{u}}{q}{m}
\qf{d_{1},\dots,d_{v}}{q}{n}
} \frac{x^{m} y^{n} q^{i \binom{m}{2}+j\binom{n}{2}+kmn}}{\qf{q}{q}{m}\,\qf{q}{q}{n}}.
\end{multline*}
\end{rem}

{\bf Limit relations.} Let us consider the case when $a=q^{\alpha}$, $b=q^{\beta}$, 
$c=q^{\gamma}$ and $d=-q^{\delta}$.  As $q \uparrow 1$ the second-order partial 
$q$-difference equation goes formally to the following second-order partial differential 
equation of the hypergeometric type 
\begin{multline}\label{eq:genappell}
\left(x^2-1 \right) \frac{\partial^{2}}{\partial x^{2}} f(x,y)+
\left(y^2-1 \right) \frac{\partial^{2}}{\partial y^{2}} f(x,y) +
2 \left((x+1) (y-1) \right) \frac{\partial^{2}}{\partial x \partial y} f(x,y) \\
+ \left(x (\alpha +\beta +\gamma +3)+\alpha -\beta +\gamma +1 \right) \frac{\partial}{\partial x} f(x,y) \\
+ \left(y (\alpha +\beta +\gamma +3)+\alpha -\beta -\gamma -1 \right) \frac{\partial}{\partial y} f(x,y) \\
-n (\alpha +\beta +\gamma +n+2) f(x,y)=0.
\end{multline}
An orthogonality weight function for the polynomial solutions of the above equation can be computed 
in the same way as in \cite{MR2853206}, giving rise to
\begin{equation}\label{pesojac}
\varrho^{(\alpha,\beta,\gamma)}(x,y)=(1-y)^{\alpha } (x+1)^{\beta } (y-x)^{\gamma },
\end{equation}
in the triangular domain 
\begin{equation}\label{domainjac}
{\mathcal{R}}=\{(x,y) \in {\mathbf{R}}^{2} \, \vert \, x \leq y \leq 1, \,\,\, -1 \leq x \leq 1 \}.
\end{equation}
It is important to note that
\begin{equation*}
\lim_{q \uparrow 1} W(x,y;q^\alpha,q^\beta,q^\gamma,-q^\delta;q)
   = (1-y)^{\alpha } (x+1)^{\beta } (y-x)^{\gamma }=\varrho^{(\alpha,\beta,\gamma)}(x,y).
\end{equation*}
The monic polynomial solutions of (\ref{eq:genappell}) satisfy a three-term recurrence relation 
of the form  (\ref{RRTT}), where the matrix coefficients can be easily computed by considering 
the limit 
as $q \uparrow 1$ in (\ref{eq:bn1})---(\ref{eq:cn2}) for $a=q^{\alpha}$, $b=q^{\beta}$, 
$c=q^{\gamma}$ and $d=-q^{\delta}$, or, eventually, from \cite{AGR2012}. The monic 
orthogonal polynomial solutions of (\ref{eq:genappell}) can be written in terms of generalized 
Kamp\'e de F\'eriet hypergeometric series as
\begin{multline}\label{ferietappell}
{\widehat{\mathrm{A}}}_{n,m}^{(\alpha,\beta,\gamma)}(x,y)=(-1)^{n} 2^{n+m}\frac{(\alpha+1)_{m}
\,(\beta+1)_{n}}{(\alpha +\beta +\gamma +m+n+2)_{n+m}}\,
\\ \times \feriet{1:1;1}{0:1;1}{\alpha +\beta +\gamma +m+n+2}{-n;-m}{-}{\beta+1;\alpha+1}
{\frac{x+1}{2}}{\frac{1-y}{2}}.
\end{multline}

\begin{rem}
We would like to mention the following limit relation between the monic bivariate big 
$q$-Jacobi polynomials in (\ref{eq:solhyp}) and the monic bivariate Jacobi polynomials, 
defined in (\ref{ferietappell}), 
\begin{equation}
\lim_{q \to 1} \widehat{P}_{n,m}(x,y;q^{\alpha},q^\beta,q^\gamma,-q^{\delta};q)
={\widehat{\mathrm{A}}}_{n,m}^{(\alpha,\beta,\gamma)}(x,y).
\end{equation}
\end{rem}

\begin{rem}
Note the following limit relation for the non-monic bivariate big $q$-Jacobi polynomials, 
defined in (\ref{eq:bbqj}),
\begin{multline*}
\lim_{q \to 1} P_{n,m}(x,y;q^{\alpha},q^{\beta},q^{\gamma},-q^{\delta};q)
=(y+1)^m \, \hyper{2}{1}{-m,\beta +\gamma +m+1}{\gamma   +1}{\frac{y-x}{y+1}} \\
\times \hyper{2}{1}{m-n,\alpha +\beta+\gamma +m+n+2}{\alpha +1}{\frac{1-y}{2}} 
=:J_{n,m}(x,y;\alpha,\beta,\gamma), \quad 0 \leq m \leq n.
\end{multline*}
The polynomials $J_{n,m}(x,y;\alpha,\beta,\gamma)$ are a non-monic polynomial solution 
of  (\ref{eq:genappell}) and they are orthogonal on the same domain (\ref{domainjac}) 
with respect to the weight (\ref{pesojac}). This non-monic polynomial solution can be written 
as
\[
J_{n,m}(x,y;\alpha,\beta,\gamma)=\frac{m! (y+1)^m (n-m)! }{(\gamma +1)_m (\alpha +1)_{n-m}} 
P_m^{(\gamma ,\beta )}\left(\frac{2 x-y+1}{y+1}\right) P_{n-m}^{(\alpha ,\beta +\gamma +2 m+1)}(y),
\]
where
\[
P_{n}^{(a,b)}(x)=\frac{(a+1)_{n}}{n!} \hyper{2}{1}{-n,n+a+b+1}{a+1}{\frac{1-x}{2}}, 
\qquad a>-1, \quad b>-1,
\]
are the Jacobi polynomials \cite[Eq. (9.8.1)]{MR2656096}.
\end{rem}

There exists at least a third family of orthogonal polynomial solutions of the partial differential equation 
(\ref{eq:genappell}) (on the same domain ${\mathcal{R}}$, defined in (\ref{domainjac}), and with respect 
to the weight function $\varrho^{(\alpha,\beta,\gamma)}$, given in (\ref{pesojac})). The non-monic 
polynomials which can be computed from the Rodrigues' formula \cite[Eq. (36)]{AGR2012},
\begin{equation}\label{Eq:108}
\tilde{A}_{n,m}^{(\alpha,\beta,\gamma)}(x,y)
=\frac{1}{\varrho^{(\alpha,\beta,\gamma)}(x,y)} \frac{\partial^{n+m}}{\partial x^{n} \, 
\partial y^{m}} \left[ (x+1)^{\beta+n}\,(1-y)^{\alpha+m}\,(y-x)^{\gamma+n+m} \,\right] \,.
\end{equation}

\begin{rem}
Observe the following limit relation between the non-monic bivariate big $q$-Jacobi polynomials, 
derived from the Rodrigues' formula (\ref{eq:rodrigues}), and  the non-monic bivariate Jacobi 
polynomials,  defined by the Rodrigues' formula (\ref{Eq:108}),
\begin{equation}
\lim_{q \to 1} \tilde{P}_{n,m}(x,y;q^{\alpha},q^\beta,q^\gamma,-q^{\delta};q)
=\tilde{A}_{n,m}^{(\alpha,\beta,\gamma)}(x,y).
\end{equation}
\end{rem}

\section{Conclusions and an outlook of future research}

In the present work we have initiated a general approach to the study of solutions of bivariate linear 
second-order partial $q$-difference equations on non-uniform lattices and concentrated our efforts 
on the particular case of $q$-linear lattices of the form $x(s)=q^{s}$ and $y(t)=q^{t}$. We have 
dealt with those bivariate polynomials, written in vector representation (and graded lexicographical order), 
that are solutions of admissible potentially self-adjoint linear second-order partial $q$-difference equation 
of the  hypergeometric type. In this context, we have proved that (similar to the one variable 
hypergeometric-type case) the coefficients of the three-term recurrence relations, obeyed by 
the vector polynomials, can be written explicitly in terms of the coefficients of the partial $q$-difference 
equation, they satisfy. It has been shown that the orthogonality weight function is completely determined 
by the coefficients of the partial $q$-difference equation.

The results obtained suggest that this approach can be extended to the quadratic lattices; this direction
has been already taken up in \cite{tesisjaime} and will be a subject of future research.

\section*{Acknowledgments}

This work has been partially supported by the Ministerio de Econom\'{\i}a y Competi\-tividad of Spain 
under grants MTM2009--14668--C02--01 and MTM2012--38794--C02--01, co-financed by the European 
Community fund FEDER. The participation of NA in this work has been supported by the DGAPA-UNAM
IN105008-3 and SEP-CONACYT 79899 projects ``\'Optica Matem\'atica''.


\end{document}